\documentclass[reqno, 10pt]{amsart}
\usepackage{amssymb}
\usepackage[usenames, dvipsnames]{color}
\usepackage{stmaryrd}
\usepackage{mathrsfs}
\usepackage{esint}

\usepackage{todonotes}
\usepackage{enumerate}
\usepackage{hyperref}

\usepackage{verbatim}

\theoremstyle{plain}
\newtheorem{theorem}{Theorem}[section]
\newtheorem{lemma}[theorem]{Lemma}
\newtheorem{corollary}[theorem]{Corollary}
\newtheorem{proposition}[theorem]{Proposition}

\theoremstyle{definition}
\newtheorem{definition}[theorem]{Definition}

\newtheorem{assumption}[theorem]{Assumption}

\theoremstyle{remark}
\newtheorem{remark}[theorem]{Remark}

\newcommand{\supp}{\operatorname{supp}}

\renewcommand{\chi}{\mathbf{M}}

\numberwithin{equation}{section}

\newcommand{\bC}{\mathbb{C}}
\newcommand{\bN}{\mathbb{N}}

\newcommand{\bR}{\mathbb{R}}
\newcommand{\bH}{\mathbb{H}}
\newcommand{\bZ}{\mathbb{Z}}

\newcommand\cD{\mathcal{D}}

\newcommand\cH{\mathcal{H}}

\newcommand\cM{\mathcal{M}}

\newcommand\cQ{\mathcal{Q}}

\newcommand\cW{\mathcal{W}}

\newcommand\sL{\mathscr{L}}

\newcommand\sW{\mathscr{W}}

\providecommand{\norm}[1]{\lVert#1\rVert}

\makeatletter
\def\dashint{\operatorname%
{\,\,\text{\bf--}\kern-.98em\DOTSI\intop\ilimits@\!\!}}
\makeatother

\begin{document}
\title[Nondivergence form degenerate parabolic equations]{Nondivergence form degenerate linear parabolic equations on the upper half space}

\author[H. Dong]{Hongjie Dong}
\address[H. Dong]{Division of Applied Mathematics, Brown University, 182 George Street, Providence RI 02912, USA}
\email{hongjie\_dong@brown.edu}

\author[T. Phan]{Tuoc Phan}
\address[T. Phan]{Department of Mathematics, University of Tennessee, 227 Ayres Hall,
1403 Circle Drive, Knoxville, TN 37996-1320, USA}
\email{phan@math.utk.edu}

\author[H. V. Tran]{Hung Vinh Tran}
\address[H. V. Tran]{Department of Mathematics, University of Wisconsin-Madison, Van Vleck Hall
480 Lincoln Drive
Madison, WI  53706, USA}
\email{hung@math.wisc.edu}

\thanks{
H. Dong is partially supported by the NSF under agreement DMS-2055244 and the Simons Fellows Award 007638.
H. Tran is supported in part by NSF CAREER grant DMS-1843320 and a Vilas Faculty Early-Career Investigator Award.
}
\subjclass[2020]{35K65, 35K67, 35K20, 35D30}
\keywords{Degenerate  linear parabolic equations; degenerate viscous Hamilton-Jacobi equations; nondivergence form; boundary regularity estimates; existence and uniqueness; weighted Sobolev spaces}

\begin{abstract}
We study a class of nondivergence form second-order degenerate linear parabolic equations in $(-\infty, T) \times \bR^d_+$ with the homogeneous Dirichlet boundary condition on $(-\infty, T) \times \partial \bR^d_+$, where $\bR^d_+ = \{x =(x_1,x_2,\ldots, x_d) \in \bR^d\,:\, x_d>0\}$ and $T\in {(-\infty, \infty]}$ is given.
The coefficient matrices of the equations are the product of $\mu(x_d)$ and bounded positive definite matrices, where $\mu(x_d)$ behaves like $x_d^\alpha$ for some given $\alpha \in (0,2)$, which are degenerate on the boundary $\{x_d=0\}$ of the domain.
Under a partially weighted VMO assumption on the coefficients, we obtain the wellposedness and regularity of solutions in weighted Sobolev spaces.
The results are applied to study the regularity theory of solutions to a class of degenerate viscous Hamilton-Jacobi equations.
 \end{abstract}

\dedicatory{Dedicated to Professor Mikhail Safonov on the occasion of his $70^{\text{th}}$ birthday}
\maketitle

\section{Introduction and main results}

\subsection{Settings}
Let $T\in (-\infty,\infty]$ and $\Omega_T=(-\infty,T)\times \bR^d_+$ with $\bR^d_+ = \bR^{d-1} \times \bR_+$ for $d \in \mathbb{N}$ and $\bR_+ = (0,\infty)$.
We study the following degenerate parabolic equation in nondivergence form
\begin{equation}\label{eq:main}
\begin{cases}
\sL u=\mu(x_d) f \quad &\text{ in } \Omega_T,\\
u=0 \quad &\text{ on } (-\infty, T) \times \partial \bR^d_+,
\end{cases}
\end{equation}
where $u: \Omega_T \rightarrow \bR$ is an unknown solution, $f: \Omega_T \rightarrow \bR$ is a given measurable forcing term, and
\begin{equation} \label{L-def}
\sL u =  a_0(z) u_t+\lambda c_0(z)u-\mu(x_d) a_{ij}(z)D_i D_j u.
\end{equation}
Here in \eqref{L-def}, $\lambda\ge 0$ is a constant, $z=(t,x) \in \Omega_T$ with $x = (x', x_d) \in \bR^{d-1} \times \bR_+$, $D_i$ denotes the partial derivative with respect to $x_i$, and $a_0, c_0: \Omega_T \rightarrow \bR$ and  $\mu: \bR_+ \rightarrow \bR$ are measurable and satisfy
\begin{equation} \label{con:mu}
 a_0(z), \ c_0(z), \ \frac{\mu(x_d)}{x_d^\alpha}  \in[\nu,\nu^{-1}], \quad \forall \ x_d \in \bR_+, \quad \forall \ z \in \Omega_T,
\end{equation}
for some given $\alpha\in (0,2)$ and $\nu \in (0,1)$. Moreover, $(a_{ij}): \Omega_T \rightarrow \mathbb{R}^{d\times d}$ are measurable and satisfy the uniform ellipticity and boundedness conditions
\begin{equation} \label{con:ellipticity}
\nu|\xi|^2 \leq a_{ij}(z) \xi_i \xi_j, \quad |a_{ij}(z)| \leq \nu^{-1}, \quad \forall \ z \in \Omega_T,
\end{equation}
for all $\xi = (\xi_1, \xi_2, \ldots, \xi_d) \in \bR^d$.

\smallskip
We observe that due to \eqref{con:mu} and \eqref{con:ellipticity}, the diffusion coefficients in the PDE in \eqref{eq:main} are degenerate when $x_d \rightarrow 0^+$, and singular when $x_d \rightarrow \infty$.  We also note that the PDE in \eqref{eq:main} can be written in the form
\[
[a_0(z) u_t + \lambda c_0(z)  u]/\mu(x_d) - a_{ij}(z) D_iD_j u = f \quad \text{in} \quad \Omega_T,
\]
in which the singularity and degeneracy appear in the coefficients of the terms involving $u_t$ and $u$. In the special case when $a_0 = c_0 =1, \mu(x_d) = x_d^\alpha$, and $(a_{ij})$ is an identity matrix, the equation \eqref{eq:main} is reduced to
\begin{equation} \label{simplest-eqn}
\left\{
\begin{array}{cccl}
u_t + \lambda u - x_d^\alpha \Delta u & = & x_d^\alpha f  & \quad \text{in} \quad \Omega_T, \\
 u & =& 0 & \quad \text{on} \quad (-\infty, T) \times \partial \bR^d_+,
\end{array} \right.
\end{equation}
in which the results obtained in this paper are still new.

\smallskip
The theme of this paper is to study the existence, uniqueness, and regularity estimates for solutions to \eqref{eq:main}. To demonstrate our results, let us state the following theorem which gives prototypical estimates of our results in a special weighted Lebesgue space $L_p(\Omega, x_d^\gamma\, dz)$ with the power weight $x_d^\gamma$ and norm
\[
\|f\|_{L_p(\Omega, x_d^\gamma dz)} = \left( \int_{\Omega_T} |f(t,x)|^p x_d^\gamma\, dx dt \right)^{1/p}.
\]
For any measurable function $f$ and $s \in \bR$, we define the multiplicative operator $(\chi^s f)(\cdot)=x_d^s f(\cdot)$.
\begin{theorem} \label{thm:demo}
Let $\alpha \in (0,2), \lambda >0$, $p\in (1,\infty)$, and $\gamma \in \big(p(\alpha-1)_+-1,2p-1\big)$. Then, for every $f \in L_p(\Omega, x_d^\gamma\, dz)$, there exists a unique strong solution $u$ to \eqref{simplest-eqn}, which satisfies
\begin{align} \label{show-est-1}
\|\chi^{-\alpha}u_t\|_{L_p}+\lambda\|\chi^{-\alpha}u\|_{L_p}
+\|D^2 u\|_{L_p}\le N\|f\|_{L_p}
\end{align}
and additionally
\begin{equation}\label{show-est-2}
\lambda^{1/2}\|\chi^{-\alpha/2}Du\|_{L_p}
\le N\|f\|_{L_p} \quad \text{if} \quad \gamma\in (\alpha p/2-1,2p-1),
\end{equation}
where $\|\cdot\|_{L_p} = \|\cdot \|_{L_p(\Omega, x_d^\gamma dz)}$ and $N=N(d,\nu,\alpha, \gamma, p)>0$. If $\frac{d+2+\gamma_+}{p}<1$, then the solution $u$ is also in $C^{{(1+\beta)}/2, 1+\beta}((-\infty, T) \times \overline{\bR}^{d}_+)$ with $\beta = 1- \frac{d+2+\gamma_+}{p}$.
\end{theorem}
\noindent
See Corollary \ref{cor1} and Theorem \ref{thm:xd} for more general results.
We note that the ranges of $\gamma$ in \eqref{show-est-1}--\eqref{show-est-2} are optimal as pointed out in Remarks \ref{remark-1-range}--\ref{remark-2-range} below.
In fact, in this paper, a much more general result in weighted mixed-norm spaces is established in Theorem \ref{main-thrm}.
As an application, we obtain a regularity result for solutions to degenerate viscous Hamilton-Jacobi equations in Theorem \ref{example-thrm}.
To the best of our knowledge, our main results (Theorems \ref{thm:demo}, \ref{main-thrm}, \ref{thm:xd}, Corollary \ref{cor1}, and Theorem \ref{example-thrm}) appear for the first time in the literature.

\subsection{Relevant literature}
The literature on regularity theory for solutions to degenerate elliptic and parabolic equations is extremely rich, and we only describe results related to \eqref{eq:main}.

\smallskip
The divergence form of \eqref{eq:main} was studied by us in \cite{DPT21} with motivation from the regularity theory of solutions to degenerate viscous Hamilton-Jacobi equations of the form
\begin{equation}\label{eq:HJ-intro}
u_t+\lambda u-\mu(x_d) \Delta u=H(z,Du) \qquad \text{ in } \Omega_T.
\end{equation}
Here, $H:\Omega_T \times \bR^d \to \bR$ is a given Hamiltonian.
Under some appropriate conditions on $H$, we obtain a regularity {and solvability} result for 
\eqref{eq:HJ-intro} in Theorem \ref{example-thrm}.
Another class of divergence form equations, which is closely related to that in \cite{DPT21}, was analyzed recently in \cite{JinXiong2} 
when $\alpha<1$. When $\alpha=2$ and $ d=1$, a specific version of \eqref{eq:HJ-intro} gives the well-known  Black-Scholes-Merton PDE that appears in mathematical finance. The analysis for \eqref{eq:main} when $\alpha\geq 2$ is completely open.

\smallskip
A similar equation to \eqref{eq:main}, \eqref{simplest-eqn}, and \eqref{eq:HJ-intro}
\[
 u_t+\lambda u-\beta D_du  - x_d \Delta u = f \qquad \text{ in } \Omega_T
\]
with an additional structural condition $\beta>0$, an important prototype equation in the study of porous media equations and parabolic Heston equation, was studied extensively in the literature (see \cite{DaHa, FePo, Koch, JinXiong1, JinXiong2} and the references therein).
We stress that we do not require this structural condition in the analysis of \eqref{eq:main} and \eqref{eq:HJ-intro}, and thus, our analysis is rather different from those in \cite{DaHa, FePo, Koch}.

\smallskip
We note that similar results on the wellposedness and regularity estimates in weighted Sobolev spaces for a different class of equations with singular-degenerate coefficients were established in a series of papers \cite{DP-20, DP-21, DP-JFA, DP-AMS}.
There, the weights of singular/degenerate coefficients of $u_t$ and $D^2u$ appear in a balanced way, which plays a crucial role in the analysis and functional space settings.
If this balance is lost, then Harnack's inequalities were proved in \cite{Chi-Se-1, Chi-Se-2} to be false in certain cases. However, with an explicit weight $x_d^\alpha$ as in our setting, it is not known if some version of Harnack's inequalities and H\"{o}lder estimates of the Krylov-Safonov type as in \cite{K-S} still hold for in \eqref{eq:main}.
Of course, \eqref{eq:main} does not have this balance structure, and our analysis is quite different from those in \cite{DP-20, DP-21, DP-JFA, DP-AMS}.

\smallskip
Finally, we emphasize again that the literature on equations with singular-degenerate coefficients is vast. Below, let us give some references on other closely related results.
The H\"{o}lder regularity 
for solutions to elliptic equations with singular and degenerate coefficients, which are in the $A_2$-Muckenhoupt {class}, were proved in the classical papers \cite{Fabes, FKS}.  See also the books \cite{Fichera, OR}, 
the papers \cite{KimLee-Yun, Sire-1, Sire-2}, and the references therein for other results on the wellposedness,  H\"{o}lder, and Schauder regularity estimates for various  classes of degenerate equations. Note also that the Sobolev regularity theory version of the 
results in \cite{Fabes, FKS} was developed and proved in \cite{Men-Phan}. In addition, we would like to point out that equations with degenerate coefficients also appear naturally in geometric analysis  \cite{Lin, WWYZ}, in which H\"{o}lder and Schauder estimates for solutions were proved.

\subsection{Main ideas and approaches}
The main ideas of this paper are along the lines with those in \cite{DPT21}.
However, at the technical level, the proofs of our main results are quite different from those in \cite{DPT21}. More precisely,  instead of the $L_2$-estimates as in \cite{DPT21}, the starting point in this paper is the weighted $L_p$-result in Lemma \ref{l-p-sol-lem} which is based on the weighted $L_p$ for divergence form equations established in \cite{DPT21}, an idea introduced by Krylov \cite{Kr99}, together with a suitable scaling.  Moreover, while the proofs in \cite{DPT21} use the Lebesgue measure as an underlying measure, in this paper we make use of more general underlying measure $\mu_1 (dz) = x_d^{\gamma_1}$ with an appropriate parameter $\gamma_1$. In particular, this allows us to obtain an optimal range of exponents for power weights in Corollary \ref{cor1}. See Remarks \ref{remark-1-range} - \ref{remark-2-range}. Several new H\"{o}lder 
estimates for higher order derivatives of solutions 
to a class of degenerate homogeneous equations  are  proved in Subsections \ref{subsec:boundary}--\eqref{subsec:int}. The results and techniques developed in these subsections 
{might be} of independent interest.

\subsection*{Organization of the paper}
The paper is organized as follows.
In Section \ref{sec:2}, we introduce various function spaces, assumptions, and  then state our main results.
The filtration of partitions, a quasi-metric, the weighted mixed-norm Fefferman-Stein theorem and Hardy-Littlewood theorem are recalled in Section \ref{preli}. A weighted parabolic embedding result is also proved in this section.
Then, in Section \ref{sec:3}, we consider \eqref{eq:main} in the case when the coefficients in \eqref{eq:main} only depend on the $x_d$ variable.
A special version of Theorem \ref{main-thrm}, Theorem \ref{thm:xd}, will be stated and proved in this section.
The proofs of Theorem \ref{main-thrm} and Corollary \ref{cor1} are given in Section \ref{sec:4}.
Finally, we study the degenerate viscous Hamilton-Jacobi equation \eqref{eq:HJ-intro} in Section \ref{sec:5}.
\section{Function spaces, parabolic cylinders, and main results}\label{sec:2}
\subsection{Function spaces}
Fix $p,q \in [1, \infty)$,  $-\infty\le S<T\le +\infty$, and a domain $\cD \subset \bR^d_+$. Denote by $L_p((S,T)\times \cD)$ the usual Lebesgue space consisting of measurable functions $u$ on $(S,T)\times \cD$ such that
\[
\|u\|_{L_p( (S,T)\times \cD)}= \left( \int_{(S,T)\times \cD} |u(t,x)|^p\, dxdt \right)^{1/p} <\infty.
\]
For a given weight $\omega$ on $(S,T)\times \cD$, let $L_{p}((S,T)\times \cD,\omega)$ be the weighted Lebesgue space on $(S,T)\times \cD$ equipped with the norm
\begin{equation*}
\|u\|_{L_{p}((S,T)\times \cD, \omega)}=\left(\int_{(S,T)\times \cD} |u(t,x)|^p \omega (t,x)\, dx dt\right)^{1/p}   < \infty.
\end{equation*}
For the weights $\omega_0=\omega_0(t)$, $\omega_1=\omega_1(x)$, and a measure $\sigma$ on $\cD$, set $\omega(t,x)=\omega_0(t)\omega_1(x)$ and define $L_{q,p}((S,T)\times \cD,\omega d\sigma)$ to be the weighted and mixed-norm Lebesgue space on $(S,T)\times \cD$ equipped with the norm
\[
\|u\|_{L_{q,p}((S,T)\times \cD, \omega d\sigma)}=\left(\int_S^T \left(\int_{\cD} |u(t,x)|^p \omega_1(x)\, \sigma(dx)\right)^{q/p} \omega_0(t)\,dt \right)^{1/q}   < \infty.
\]

\subsubsection{Function spaces for nondivergence form equations}
Consider $\alpha>0$. We define the solution spaces as follows.
Firstly, define
\[
 W_{p}^{1,2}((S,T)\times \cD, \omega)
 =\left\{u \,:\,  \chi^{-\alpha} u, \chi^{-\alpha} u_t, D^2u \in  L_p((S,T) \times \cD,\omega)\right\},
\]
where, for $u\in  W_{p}^{1,2}((S,T)\times \cD, \omega)$,
\begin{multline*}
\|u\|_{W^{1,2}_p((S,T)\times \cD,\omega)}\\
=\| \chi^{-\alpha} u\|_{L_p((S,T)\times \cD,\omega)}+\| \chi^{-\alpha} u_t\|_{L_p((S,T)\times \cD,\omega)}+\|D^2u\|_{L_p((S,T)\times \cD,\omega)}
\end{multline*}
and for $s \in \mathbb{R}$ the multiplicative operator $\chi^s$ is defined as $\chi^s f(\cdot)=x_d^s f(\cdot)$.

\smallskip
Let $\sW^{1,2}_p((S,T)\times \cD,\omega)$ be the closure in $W^{1,2}_p((S,T)\times \cD,\omega)$ of all compactly supported functions in $C^\infty((S,T)\times \overline{\cD})$ vanishing near $\overline{\cD} \cap \{x_d=0\}$ if $\overline{\cD} \cap \{x_d=0\}$ is not empty.
The space $\sW^{1,2}_p((S,T)\times \cD,\omega)$ is equipped with the same norm $\|\cdot\|_{\sW^{1,2}_p((S,T)\times \cD,\omega)}=\|\cdot\|_{W^{1,2}_p((S,T)\times \cD,\omega)}$.
When there is no time dependence, we write these two spaces as $W^2_p(\cD,\omega)$ and $\sW^2_p(\cD,\omega)$, respectively.

\smallskip
Next, denote by
\[
\begin{split}
& W_{q,p}^{1,2}((S,T)\times \cD, \omega\, d\sigma)\\
& \qquad =\left\{u \,:\,  \chi^{-\alpha} u, \chi^{-\alpha} u_t,D^2u \in  L_{q,p}((S,T) \times \cD,\omega\, d\sigma)\right\},
\end{split}
\]
which is equipped with the norm
\begin{multline*}
\|u\|_{W^{1,2}_{q,p}((S,T)\times \cD,\omega\, d\sigma)}
=\| \chi^{-\alpha} u\|_{L_{q,p}((S,T)\times \cD,\omega\, d\sigma)}\\
+\| \chi^{-\alpha} u_t\|_{L_{q,p}((S,T)\times \cD,\omega\, d\sigma)}+\|D^2u\|_{L_{q,p}((S,T)\times \cD,\omega\, d\sigma)}.
\end{multline*}
Let $\sW^{1,2}_{q,p}((S,T)\times \cD,\omega\, d\sigma)$ be the closure in $W^{1,2}_{q,p}((S,T)\times \cD,\omega d\sigma)$ of all compactly supported functions in $C^\infty((S,T)\times \overline{\cD})$ vanishing near $\overline{\cD} \cap \{x_d=0\}$ if $\overline{\cD} \cap \{x_d=0\}$ is not empty.
The space $\sW^{1,2}_{q,p}((S,T)\times \cD,\omega\, d\sigma)$ is equipped with the same norm $\|\cdot\|_{\sW^{1,2}_{q,p}((S,T)\times \cD,\omega\, d\sigma)}=\|\cdot\|_{W^{1,2}_{q,p}((S,T)\times \cD,\omega\, d\sigma)}$.

\subsubsection{Function spaces for divergence form equations}
We also need function spaces for divergence form equations in this paper, which are taken from \cite{DPT21}.
Set
$$
W^1_p((S,T)\times \cD, \omega)=\left\{u\,:\, \chi^{-\alpha/2} u, Du\in L_p((S,T)\times \cD, \omega)\right\},
$$
which is equipped with the norm
$$
\|u\|_{W^1_p((S,T)\times \cD,\omega)}=\| \chi^{-\alpha/2} u\|_{L_p((S,T)\times \cD,\omega)}+\|Du\|_{L_p((S,T)\times \cD,\omega)}.
$$
We denote by $\sW^1_p((S,T)\times \cD,\omega)$ the closure in $W^1_p((S,T)\times \cD,\omega)$ of all compactly supported functions in $C^\infty((S,T)\times \overline{\cD})$ vanishing near $\overline{\cD} \cap \{x_d=0\}$ if $\overline{\cD} \cap \{x_d=0\}$ is not empty.
The space $\sW^1_p((S,T)\times \cD,\omega)$ is equipped with the same norm $\|\cdot\|_{\sW^1_p((S,T)\times \cD,\omega)}=\|\cdot\|_{W^1_p((S,T)\times \cD,\omega)}$.

\smallskip
Set
 \[
\begin{split}
& \bH_{p}^{-1}( (S,T)\times \cD, \omega) \\
& =\big\{u\,:\, u  =  \mu(x_d) D_iF_i +f_1+f_2, \ \ \text{where}\  \chi^{1-\alpha} f_1,\chi^{-\alpha/2}f_2\in L_{p}( (S,T)\times \cD, \omega)\\
& \qquad\text{and }
 F= (F_1,\ldots,F_d) \in L_{p}((S,T)\times \cD, \omega)^{d}\big\},
\end{split}
\]
equipped with the norm
\begin{align*}
&\|u\|_{\bH_{p}^{-1}((S,T)\times \cD, \omega)} \\
&=\inf\big\{\|F\|_{L_{p}((S,T)\times \cD, \omega)}
+\|| \chi^{1-\alpha} f_1|+|\chi^{-\alpha/2}f_2|\|_{L_{p}((S,T)\times \cD, \omega)}\,:\\
&\qquad u= \mu(x_d) D_iF_i +f_1+f_2\big\}.
\end{align*}
Define
\[
 \cH_{p}^1((S,T)\times \cD, \omega)
 =\big\{u \,:\,  u \in  \sW^1_p((S,T) \times \cD,\omega)),
 u_t\in  \bH_{p}^{-1}( (S,T)\times \cD, \omega)\big\},
\]
where, for $u\in   \cH_{p}^1((S,T)\times \cD, \omega)$,
\begin{align*}
\|u\|_{\cH_{p}^1((S,T)\times \cD, \omega)} &= \|\chi^{-\alpha/2} u\|_{L_{p}((S,T)\times \cD, \omega)} + \|Du\|_{L_{p}((S,T)\times \cD, \omega)} \\
& \qquad +\|u_t\|_{\bH_{p}^{-1}((S,T)\times \cD, \omega)}.
\end{align*}

\subsection{Parabolic cylinders}
We use the same setup as that in \cite{DPT21}.
For $x_0 = (x_0', x_{0d}) \in \bR^{d-1} \times \bR_+$ and $\rho>0$, denote by $B_\rho(x_0)$ the usual ball with center $x_0$ radius $\rho$ in $\bR^d$, $B_\rho'(x_0')$ the ball center $x_0'$ radius $\rho$ in $\bR^{d-1}$, and
\[
B_\rho^+(x_0) = B_\rho(x_0) \cap \bR^d_+.
\]
We note that \eqref{eq:main} is invariant under the scaling
\begin{equation} \label{scaling}
(t,x) \mapsto (s^{2-\alpha} t, sx), \qquad s > 0.
\end{equation}
For $x_d \sim x_{0d} \gg 1$, $a_{ij} = \delta_{ij}$, and $\lambda =f=0$, then \eqref{eq:main}  behaves like a heat equation
\[
u_t  -x_{0d}^{\alpha} \Delta u = 0,
\]
which can be reduced to the heat equation with unit heat constant under the scaling
\[
(t,x) \mapsto (s^{2-\alpha} t, s^{1-\alpha/2} x_{0d}^{-\alpha/2}x), \quad s>0.
\]
It is thus natural to use the following parabolic cylinders in $\Omega_T$ in this paper.
For $z_0 = (t_0, x_0) \in (-\infty, T) \times \bR^d_+$ with $x_0= (x_0', x_{0d}) \in \bR^{d-1} \times \bR_+$ and $\rho>0$, set
\begin{equation} \label{def:Q}
\begin{split}
& Q_{\rho}(z_0) =  (t_0 - \rho^{2-\alpha}, t_0) \times B_{r(\rho, x_{0d})} (x_0), \quad \\
&Q_{\rho}^+(z_0) = Q_\rho(z_0) \cap \{x_d>0\},
\end{split}
\end{equation}
where
\begin{equation} \label{def:r}
r(\rho,x_{0d}) = \max\{\rho, x_{0d}\}^{\alpha/2} \rho^{1-\alpha/2}.
\end{equation}
Of course, $Q_{\rho}(z_0) = Q_{\rho}^+(z_0) \subset (-\infty, T) \times \bR^d_+$ for $\rho \in (0,x_{0d})$. For $z' = (t, x') \in \bR \times \bR^{d-1}$, we write
\[
Q_{\rho}'(z') = (t-\rho^{2-\alpha}, t_0) \times B_{\rho}'(x').
\]
Finally, when $x_0 =0, t_0=0$, for simplicity of notation we drop $x_0, z_0$ and write $B_\rho=B_\rho(0), Q_\rho=Q_\rho(0)$, and $Q_\rho^+=Q_\rho^+(0)$, etc.
\subsection{Main results} \label{main-result-sect} Throughout the paper, for a locally integrable function $f$, a locally finite measure $\omega$, and a domain $Q\subset \bR^{d+1}$, we write
\begin{equation} \label{everage-def}
(f)_{Q} = \fint_{Q} f(s,y)\, dyds, \qquad (f)_{Q,\omega} = \frac{1}{\omega(Q)}\int_{Q} f(s,y) \,\omega(dyds).
\end{equation}
Also, for a number $\gamma_1 \in (-1,  \infty)$ to be determined, we define
\[
\mu_1(dz) = x_d^{\gamma_1}\, dxdt.
\]
We impose the following assumption  on the partial mean oscillations of the coefficients $(a_{ij})$, $a_0$, and $c_0$.
\begin{assumption}[$\rho_0,\gamma_1, \delta$] \label{assumption:osc} For every $\rho \in (0, \rho_0)$ and $z_0= (z_0', z_{0d}) \in \overline{\Omega}_T$, there exist $[a_{ij}]_{\rho, z'}, [a_{0 }]_{\rho, z'},  [c_{0 }]_{\rho, z'}: ((x_{d} -r(\rho, x_d))_+, x_d + r(\rho, x_d)) \rightarrow \bR$ such that \eqref{con:mu}--\eqref{con:ellipticity} hold on $((x_{d} -r(\rho, x_d))_+, x_d + r(\rho, x_d))$  with  $[a_{ij}]_{\rho, z'}$, $[a_{0 }]_{\rho, z'}$, $[c_{0 }]_{\rho, z'}$ in place of $(a_{ij})$, $a_0$, $c_0$, respectively, and
\begin{align*}
a_\rho^{\#}(z_0):= & \max_{1 \leq i, j \leq d}\fint_{Q_\rho^+(z_0)} | a_{ij}(z) -[a_{ij}]_{\rho,z'}(x_d)|\,  \mu_1(dz) \\
& \qquad + \fint_{Q_\rho^+(z)} | a_{0}(z) -[a_{0}]_{\rho,z'}(x_d)|\, \mu_1(dz) \\
& \qquad + \fint_{Q_\rho^+(z)} | c_{0}(z) -[c_{0}]_{\rho,z'}(x_d)|\, \mu_1(dz) < \delta.
\end{align*}
\end{assumption}
\noindent
We note that the un-weighted partial mean oscillation was introduced  in \cite{Kim-Krylov} to study a class of elliptic equations with uniformly elliptic and bounded coefficients (i.e., $\gamma_1=\alpha=0$). Note also that by dividing the equation \eqref{eq:main} by $a_{dd}$ and adjusting $\nu$, we can assume without loss of generality throughout the paper that
\begin{equation} \label{add-assumption}
a_{dd} \equiv 1.
\end{equation}

\smallskip
The theorem below is the first main result of our paper, in which the definition of the $A_p$ Muckenhoupt class of weights can be found in Definition \ref{Def-Muck-wei} below.
\begin{theorem} \label{main-thrm} Let $T \in (-\infty, \infty]$, $\nu \in (0,1)$, $p, q, K \in (1, \infty)$, $\alpha \in (0, 2)$, and  $\gamma_1 \in (\beta_0 -\alpha, \beta_0 -\alpha +1]$ for $\beta_0 \in {(\alpha-1}, \min\{1, \alpha\}]$. Then, there exist $\delta = \delta(d, \nu, p, q, K, \alpha, \beta_0, \gamma_1)>0$ sufficiently small and $\lambda_0 = \lambda_0(d, \nu, p, q, K, \alpha, \beta_0, \gamma_1)>0$ sufficiently large such that the following assertion holds. Suppose that \eqref{con:mu}, \eqref{con:ellipticity}, and \eqref{add-assumption} are satisfied, $\omega_0 \in A_q(\bR)$, $\omega_1 \in A_p(\bR^d_+, \mu_1)$ with
\[
[\omega_0]_{A_q(\bR)} \leq K \quad \text{and} \quad [\omega_1]_{A_p(\bR^d_+, \mu_1)} \leq K, \quad \text{where} \,\, \mu_1(dz) = x_d^{\gamma_1} dxdt.
\]
Suppose also that Assumption \ref{assumption:osc} $(\rho_0, \gamma_1,\delta)$ holds for some $\rho_0>0$. Then, for any function $f \in L_{q, p}(\Omega_T,  x_d^{p(\alpha-\beta_0)} \omega\, d\mu_1)$ and $\lambda \geq \lambda_0 \rho_0^{-(2-\alpha)}$, there exists a strong solution $u{\in \sW^{1,2}_{q, p}(\Omega_T,   x_d^{p(\alpha-\beta_0)}  \omega\, d\mu_1)}$ to the equation \eqref{eq:main}, which satisfies
\begin{equation} \label{main-est-1}
\|\chi^{-\alpha} u_t\|_{L_{q,p}} + \|D^2u\|_{L_{q,p}} + \lambda \|\chi^{-\alpha} u\|_{L_{q,p}} \leq N \| f\|_{L_{q,p}},
\end{equation}
where $\omega(t, x) = \omega_0(t) \omega_1(x)$ for $(t,x) \in \Omega_T$, $L_{q,p} = L_{q,p}(\Omega_T, x_d^{p(\alpha-\beta_0)} \omega \, d\mu_1)$, 
and $N = N(d, \nu, p, q, K, \alpha, \beta_0, \gamma_1)>0$. Moreover, if $\beta_0 \in {(\alpha-1}, \alpha/2]$, then it also holds that
\begin{equation} \label{main-est-2}
\begin{split}
& \|\chi^{-\alpha} u_t\|_{L_{q,p}} + \| D^2u\|_{L_{q,p}} + \lambda \|\chi^{-\alpha} u\|_{L_{q,p}}  + \lambda^{1/2} \|\chi^{-\alpha/2} Du\|_{L_{q,p}}  \\
& \leq N \|f\|_{L_{q,p}}.
\end{split}
\end{equation}
\end{theorem}

The following is an important corollary of Theorem \ref{main-thrm} in which $\omega_1$ is a power weight of the $x_d$ variable and $\beta_0$ and $\gamma_1$ are specifically chosen.
\begin{corollary} \label{cor1} Let $T \in (-\infty, \infty]$, $\nu \in (0,1)$, $p, q \in (1, \infty)$, $\alpha \in (0, 2)$, and $\gamma \in (p(\alpha-1)_+ -1, 2p-1)$. Then, there exist $\delta = \delta(d, \nu, p, q, \alpha, \gamma)>0$ sufficiently small and $\lambda_0 = \lambda_0(d, \nu, p, q,  \alpha,  \gamma)>0$ sufficiently large such that the following assertion holds. Suppose that \eqref{con:mu}, \eqref{con:ellipticity} hold and suppose also that Assumption \ref{assumption:osc} $(\rho_0, 1-(\alpha-1)_+, \delta)$ holds for some $\rho_0>0$. Then, for any $f \in L_{q, p}(\Omega_T, x_d^{\gamma} dz)$ and $\lambda \geq \lambda_0 \rho_0^{-(2-\alpha)}$, there exists a strong solution $u \in \sW^{1,2}_{q, p}(\Omega_T,  x_d^\gamma\, dz)$ to the equation \eqref{eq:main}, which satisfies
\begin{equation} \label{cor-est-1}
\|\chi^{-\alpha} u_t\|_{L_{q,p}} + \|D^2u\|_{L_{q,p}} + \lambda \|\chi^{-\alpha} u\|_{L_{q,p}} \leq N \|f\|_{L_{q,p}},
\end{equation}
where $L_{q,p} = L_{q,p}(\Omega_T,  x_d^\gamma dz)$ and $N = N(d, \nu, p, q, \alpha, \gamma)>0$. If Assumption \ref{assumption:osc} $(\rho_0, 1-\alpha/2, \delta)$ also holds and $\gamma \in (\alpha p/2 -1, 2p-1)$, then we have
\begin{equation} \label{cor-est-2}
\|\chi^{-\alpha} u_t\|_{L_{q,p}} + \|D^2u\|_{L_{q,p}} + \lambda \|\chi^{-\alpha} u\|_{L_{q,p}} +   \lambda^{1/2} \|\chi^{-\alpha/2} Du\|_{L_{q,p}} \leq N \|f\|_{L_{q,p}}.
\end{equation}
Additionally, if $\frac{d+\gamma_+}{p} + \frac{2}{q} <1$, then the solution $u$ is also in $C^{{(1+\beta)}/2, 1+\beta}((-\infty, T) \times \overline{\bR}^{d}_+)$ with $\beta = 1- \frac{d+\gamma_+}{p} -\frac{2}{q}$.
\end{corollary}

\begin{remark}
By viewing solutions to elliptic equations as stationary solutions to parabolic equations, from Theorem \ref{main-thrm} and Corollary \ref{cor1}, we derive the corresponding results for elliptic equations. Also, by using a localization technique, similar results on local boundary $W^{1,2}_{q,p}$ estimates as those in \cite[Corollary 2.10]{DP-20} can be derived.
\end{remark}

In the remarks below, we give examples showing that the ranges of $\gamma$ in \eqref{show-est-1}--\eqref{show-est-2} as well as \eqref{cor-est-1}--\eqref{cor-est-2} are optimal.
We note  that the range of $\gamma$ for the estimate of $Du$ in \eqref{show-est-2}, \eqref{main-est-2}, and  \eqref{cor-est-2} is smaller than that for $u, u_t, D^2u$ in \eqref{show-est-1}, \eqref{main-est-1}, and   \eqref{cor-est-1}.
See Remark \ref{remark-3-range} below to see the necessity of such different ranges.

\begin{remark} \label{remark-1-range}
When $\alpha \in (0,1)$,  the range $(p(\alpha-1)_+ -1, 2p-1)$ for the power $\gamma$ in \eqref{show-est-1} becomes $(-1,2p-1)$, which agrees with the range in \cite{KN} for equations with uniformly elliptic and bounded coefficients.
See also \cite{DP-JFA} and \cite{MNS}
in which a similar range of the power $\gamma$  is also used in for a class of equations of extensional type.
When $\alpha \in [1,2)$, the lower bound $p(\alpha-1)_+ -1$ for $\gamma$ in \eqref{show-est-1} is optimal.
To see this, consider an explicit example when $d=1$, $\lambda>0$, $T < \infty$, and
\[
u(t,x)=\left(x + c x^{3-\alpha}\right) \xi(x) e^{\lambda t} \qquad \text{ for } (t,x) \in \Omega_T.
\]
Here, $\xi \in C^\infty([0,\infty),[0,\infty))$ is a cutoff function such that $\xi=1$ on $[0,1]$, $\xi=0$ on $[3,\infty)$, $\|\xi'\|_{L^\infty(\bR)} \leq 1$, and
\[
c=\frac{2 \lambda}{(3-\alpha)(2-\alpha)}.
\]
Set
\[
f(t,x)=x^{-\alpha}(u_t +\lambda u) - u_{xx}.
\]
Then, $u$ solves
\[
u_t + \lambda u - x^\alpha u_{xx} = x^\alpha f \qquad \text{ in } \Omega_T
\]
and it satisfies the boundary condition
\[ u(t, 0) =0 \quad \text{for} \,\, t \in (-\infty, T).
\]
Moreover, we see that $\chi^{-\alpha}u_t, \chi^{-\alpha}u \in L_p(\Omega_T, x^\gamma)$ for $\gamma>p(\alpha-1)-1$. We claim that
\begin{equation} \label{cout-claim-1}
f \in L_p(\Omega_T, x^{p(\alpha-1)-1}) \quad \text{but} \quad \chi^{-\alpha}u_t, \chi^{-\alpha}u \notin L_p(\Omega_T, x^{p(\alpha-1)-1}).
\end{equation}
To prove the claim \eqref{cout-claim-1}, we note that
\[
\begin{split}
& \int_{\Omega_T} |x^{-\alpha}u|^p x^{p(\alpha-1)-1}\,dz
=\int_{\Omega_T} |x^{-1}u|^p x^{-1}\,dz \\
&  \geq  \int_{0}^1 \int_{-\infty}^T x^{-1} e^{p\lambda t}\,dt dx  = N\int_{0}^1 x^{-1} \,dx=\infty.
\end{split}
\]
Thus, $\chi^{-\alpha}u_t, \chi^{-\alpha}u \notin L_p(\Omega_T, x^{p(\alpha-1)-1})$.

We next note that $f(t,x)=0$ for $(t,x) \in (-\infty,T] \times [3,\infty)$, and
\[
f(t,x)=2 c\lambda x^{3-2\alpha} e^{\lambda t} \qquad \text{ for } (t,x) \in (-\infty,T] \times [0,1].
\]
From this and
\[
\int_{0}^1 \int_0^T |x^{3-2\alpha}|^p x^{p(\alpha-1)-1} e^{p \lambda t} dt dx = N \int_0^1 x^{p(2-\alpha)-1} dx < \infty,
\]
it follows that  $f\in L_p(\Omega_T, x^{p(\alpha-1)-1})$ and \eqref{cout-claim-1} is verified.
\end{remark}

\begin{remark} \label{remark-3-range}
When $\alpha \in (0,2)$, the lower bound $\alpha p/2 -1$ for $\gamma$ in \eqref{show-est-2} is optimal.
Indeed, consider the same example as that in Remark \ref{remark-1-range} above.
It is clear that $\chi^{-\alpha/2}u_x \in L_p(\Omega_T, x^\gamma)$ for $\gamma> \alpha p/2-1$.
On the other hand, $\chi^{-\alpha/2}u_x \notin L_p(\Omega_T, x^{\alpha p/2-1})$ as
\[
 \int_{\Omega_T} |x^{-\alpha/2}u_x|^p x^{\alpha p/2-1}\,dz
=\int_{\Omega_T} |u_x|^p x^{-1}\,dz \geq  \int_{0}^1 \int_{-\infty}^T x^{-1} e^{p\lambda t}\,dt dx =\infty.
\]
Besides, $f(t,x)=0$ for $(t,x) \in (-\infty,T] \times [3,\infty)$, and
\[
f(t,x)=2 c\lambda x^{3-2\alpha} e^{\lambda t} \qquad \text{ for } (t,x) \in (-\infty,T] \times [0,1].
\]
Hence, $f \in L_p(\Omega_T, x^{\alpha p/2-1})$ as
\[
\int_{0}^1 \int_0^T |x^{3-2\alpha}|^p x^{\alpha p/2-1} e^{p \lambda t} dt dx = N \int_0^1 x^{p(3-3\alpha/2)-1} dx < \infty.
\]
\end{remark}

\begin{remark} \label{remark-2-range}
We also have that the upper bound $\gamma < 2p-1$ in \eqref{show-est-1}--\eqref{show-est-2} is optimal.
Indeed, for $\gamma=2p-1$, the trace of  $W_{p}^{2}(\cD,x_d^{2p-1})$ is not well defined.
For simplicity, let $d=1$, $\cD=[0,1/2]$, and consider
\[
\phi(x) = \log (|\log x|).
\]
Then,
\[
\phi_{xx} = \frac{1}{x^2} \left( |\log x|^{-1} - |\log x|^{-2} \right).
\]
It is clear that $\phi \in W_{p}^{2}([0,1/2],x^{2p-1})$, and $\phi$ is not finite at $0$.
\end{remark}
\section{Preliminaries} \label{preli}
\subsection{A filtration of partitions and a quasi-metric} \label{Feffer}
We recall the construction of a filtration of partitions $\{\bC_n\}_{n \in \bZ}$ (i.e., dyadic decompositions) of $\bR\times \bR^d_+$ in \cite{DPT21},
which satisfies the following three basic properties (see \cite{Krylov}):
\begin{enumerate}[(i)]
\item The elements of partitions are ``large'' for big negative $n$'s and ``small''
for big positive $n$'s: for any $f\in L_{1,\text{loc}}$,
$$
\inf_{C\in \bC_n}|C|\to \infty\quad\text{as}\,\,n\to -\infty,\quad
\lim_{n\to \infty}(f)_{C_n(z)}=f(z)\quad\text{a.e.},
$$
where $C_n(z)\in \bC_n$ is such that $z\in C_n(z)$.

\item The partitions are nested: for each $n\in \bZ$, and $C \in \bC_n$, there exists a unique $C' \in \bC_{n-1}$ such that $C \subset C'$.

\item The following regularity property holds: For $n,C, C'$ as in (ii), we have
$$
|C'|\le N_0|C|,
$$
where $N_0>0$ is independent of $n$, $C$, and $C'$.
\end{enumerate}

For $s\in \bR$, denote by $\lfloor s \rfloor$  the integer part of $s$.
For a fixed $\alpha\in (0,2)$ and $n\in \bZ$, let $k_0=\lfloor -n/(2-\alpha) \rfloor$.
The partition $\bC_n$ contains boundary cubes in the form
$$
((j-1)2^{-n},j2^{-n}]\times (i_12^{k_0},(i_1+1)2^{k_0}]
\times\cdots\times (i_{d-1}2^{k_0},(i_{d-1}+1)2^{k_0}]\times (0, 2^{k_0}],
$$
where $j,i_1,\ldots,i_{d-1}\in \bZ$, and interior cubes in the form
$$
((j-1)2^{-n},j2^{-n}]\times (i_12^{k_2},(i_1+1)2^{k_2}]
\times\cdots \times (i_d2^{k_2}, (i_d+1)2^{k_2}],
$$
where $j,i_1,\ldots,i_{d}\in \bZ$ and
\begin{equation}
                    \label{eq:part1}
i_d2^{k_2}\in [2^{k_1},2^{k_1+1})\, \text{for some integer}\, k_1\ge k_0,
\quad k_2=\lfloor (-n+k_1\alpha)/2 \rfloor-1.
\end{equation}
It is clear that $k_2$ increases with respect to $k_1$ and decreases with respect to $n$.
As $k_1\ge k_0>-n/(2-\alpha)-1$, we have
$(-n+k_1\alpha)/2-1\le k_1$,
which implies $k_2\le k_1$ and $(i_d+1)2^{k_2}\le 2^{k_1+1}$.
According to \eqref{eq:part1}, we also have
$$
(2^{k_2}/2^{k_1})^2\sim 2^{-n}/(2^{k_1})^{2-\alpha},
$$
which allows us to apply the interior estimates after a scaling.

The quasi-metric  $\varrho: \Omega_\infty\times \Omega_\infty\to  [0,\infty)$ is defined as
$$
 \varrho((t,x),(s,y))=|t-s|^{1/(2-\alpha)}
+\min\big\{|x-y|,|x-y|^{2/(2-\alpha)}\min\{x_d,y_d\}^{-\alpha/(2-\alpha)}\big\}.
$$
There exists a constant $K_1=K_1(d,\alpha)>0$ such that
$$
 \varrho((t,x),(s,y))\le K_1\big(\varrho((t,x),(\hat t,\hat x))+ \varrho((\hat t,\hat x),(s,y))\big)
$$
for any $(t,x),(s,y),(\hat t,\hat x)\in \Omega_\infty$,  and $ \varrho((t,x),(s,y))=0$ if and only if $(t,x)=(s,y)$.
Besides, the cylinder $Q_\rho^+(z_0)$ defined in \eqref{def:Q} is comparable to
$$
\{(t,x)\in \Omega: t<t_0,\, \varrho((t,x),(t_0,x_0))<\rho \}.
$$
Of course, $(\Omega_T, \varrho)$ equipped with the Lebesgue measure is a space of homogeneous type and we have the above dyadic decomposition.
\subsection{Maximal functions and sharp functions} \label{sharp-function-sec}
The dyadic maximal function and sharp function of a locally integrable function $f$ and a given weight $\omega$ in $\Omega_\infty$ are defined as
\begin{align*}
\cM_{\text{dy},\omega} f(z)&=\sup_{n<\infty}\frac{1}{\omega(C_n(z))}\int_{C_n(z)\in \bC_n}|f(s,y)| \omega(s,y)\,dyds,\\
f_{\text{dy},\omega}^{\#}(z)&=\sup_{n<\infty}\frac{1}{\omega(C_n(z))}\int_{C_n(z)\in \bC_n}|f(s,y)-(f)_{C_n(z),\omega}|\omega(s,y)\,dyds.
\end{align*}
Observe that the average notation in \eqref{everage-def} is used in the above definition.  Similarly, the maximal function and sharp function over cylinders are given by
\begin{align*}
\cM_\omega f(z)&=\sup_{z\in Q^+_\rho(z_0), z_0\in \overline{\Omega_\infty}} \frac{1}{\omega(Q_\rho^+(z_0))} \int_{Q_\rho^+(z_0)}|f(s,y)|\omega(s,y)\,dyds,\\
f^{\#}_\omega(z)&=\sup_{z\in Q^+_\rho(z_0),z_0\in \overline{\Omega_\infty}}\frac{1}{\omega(Q_\rho^+(z_0))}\int_{Q_\rho^+(z_0)}|f(s,y)-(f)_{Q^+_\rho(z_0)}|\omega(s,y)\,dyds.
\end{align*}
We have, for any $z\in \Omega_\infty$,
$$
\cM_{\text{dy},\omega} f(z)\le N\cM_{\omega} f(z) \qquad \text{ and } \qquad f_{\text{dy},\omega}^{\#}(z)\le Nf^{\#}_\omega(z),
$$
where $N=N(d,\alpha)>0$.

\smallskip
We also recall the following definition of the $A_p$ Muckenhoupt class of weights.
\begin{definition}
        \label{Def-Muck-wei}
For each $p \in (1, \infty)$  and for a nonnegative Borel measure $\sigma$ on $\bR^d$, a locally integrable function $\omega :  \bR^d \rightarrow \bR_+$ is said to be in the $A_p( \bR^d,  \sigma)$ Muckenhoupt class of weights if and only if $[\omega]_{A_p(\bR^d, \sigma)} < \infty$, where
\begin{equation}
                    \label{Ap.def}
\begin{split}
& [\omega]_{A_p(\bR^d,  \sigma)} \\
& =
\sup_{\rho >0,x =(x', x_d)\in \bR^d } \bigg[\fint_{B_{\rho} (x)} \omega(y)\,  \sigma(dy) \bigg]\bigg[\fint_{B_{\rho}(x)} \omega(y)^{\frac{1}{1-p}}\,  \sigma(dy) \bigg]^{p-1}.
\end{split}
\end{equation}
Similarly, the class of weights $A_p(\bR^d_+,  \sigma)$ can be defined in the same way in which the ball $B_{\rho} (x)$ in \eqref{Ap.def} is replaced with $B_\rho^+(x)$ for $x\in \overline{\bR^d_+}$. For weights with respect to the time variable, the definition is similar with the balls replaced with intervals $(t_0 -\rho^{2-\alpha}, t_0 + \rho^{2-\alpha})$ and $\sigma(dy)$ replaced with $dt$. If $\sigma$ is a Lebesgue measure, we simply write $A_p(\bR^d_+) = A_p(\bR^d_+, dx)$ and $A_p(\bR^d) = A_p(\bR^d, dx)$.  Note that if $\omega \in A_p(\bR)$, then $\tilde{\omega} \in A_p(\bR^d)$ with $[\omega]_{A_p(\bR)} = [\tilde{\omega}]_{A_p(\bR^d)}$, where $\tilde{\omega}(x) = \omega(x_d)$ for $x = (x', x_d) \in \bR^d$. Sometimes, if the context is clear, we neglect the spatial domain and only write $\omega \in A_p$.
\end{definition}

\smallskip
The following version of the weighted mixed-norm Fefferman-Stein theorem and Hardy-Littlewood maximal function theorem can be found in \cite{Dong-Kim-18}.
\begin{theorem}  \label{FS-thm} Let $p, q \in (1,\infty)$,  $\gamma_1 \in (-1, \infty)$, $K\geq 1$, and $\mu_1(dz) = x_d^{\gamma_1}\, dxdt$. Suppose that $\omega_0\in A_q(\bR)$ and $\omega_1 \in A_p(\bR^{d}_{+},\mu_1)$ satisfy
$$
[\omega_0]_{A_q},    \,\, [\omega_{1}]_{A_p(\bR_+^d, \mu_1)}\le K.$$
Then, for any $f \in L_{q, p}(\Omega_T, \omega\, d\mu_1)$, we have
\begin{align*} 
& \|f\|_{L_{q, p}(\Omega_T,  \omega\, d\mu_1)}
{\leq N \| f^{\#}_{\text{dy},\mu_1}\|_{L_{q,p}(\Omega_T,  \omega\, d \mu_1)}}
\leq N \| f^{\#}_{\mu_1}\|_{L_{q,p}(\Omega_T,  \omega\, d \mu_1)}, \\
& \|\mathcal{M}_{\mu_1}(f)\|_{L_{q,p}(\Omega_T, \omega\, d\mu_1)} \leq N \|f\|_{L_{q, p}(\Omega_T,  \omega\,  d \mu_1)},
\end{align*}
where $N = N(d, q, p, \gamma_1, K)>0$ and $\omega(t,x) = \omega_0(t)\omega_1(x)$ for $(t,x) \in \Omega_T$.
\end{theorem}

\subsection{Weighted  parabolic Sobolev embeddings}
We denote the standard parabolic cylinders by
\[
\mathcal{Q}_\rho(t,x) = (t-\rho^2, t) \times B_\rho(x), \quad \mathcal{Q}^+_\rho(t,x) = (t-\rho^2, t) \times B_\rho^+(x).
\]
When $x=0$ and $t=0$,  we write $\mathcal{Q}_\rho = \mathcal{Q}_\rho(0,0)$ and $\mathcal{Q}^+_\rho = \mathcal{Q}_\rho^+(0,0)$. Recall that for $\gamma \in \mathbb{R}$ and $p, q \in [1, \infty)$, we say $u \in L_{q,p}(\mathcal{Q}^+_1, x_d^\gamma dz)$ if
\[
\|u\|_{L_{q,p}(\mathcal{Q}^+_1, x_d^\gamma dz)} = \left(\int_{-1}^0 \Big ( \int_{B^+_1} |u(t,x)|^p x_d^\gamma\, dx\Big)^{q/p}\, dt \right)^{1/q} < \infty.
\]
We  denote
\[
\cW^{1,2}_{q,p}(\mathcal{Q}^+_1, x_d^\gamma dz) =\big\{ u: u_t, D^2 u \in  L_{q,p}(\mathcal{Q}^+_1, x_d^\gamma) \,\, \text{and} \ u,Du \in  L_{1, \text{loc}}(\mathcal{Q}^+_1) \big\}.
\]
We prove some weighted parabolic Morrey inequalities that are needed to prove $C^{{(1+\beta)}/2, 1+\beta}$-regularity of solutions. See  \cite[Lemma 4.66]{Adams-book} and also \cite[Theorem 5.3]{RO} for similar results for the elliptic case. Let us recall that, for an open set $Q\subset \mathbb{R}^{d+1}$,
\[
\|u\|_{C^{(1+\beta)/2, 1+\beta}(Q)}=\|u\|_{L_\infty(Q)}+\|Du\|_{L_\infty(Q)}+\llbracket u\rrbracket_{C^{(1+\beta)/2, 0}(Q)}+\llbracket Du\rrbracket_{C^{\beta/2,\beta}(Q)}.
\]
Here, $\llbracket \cdot \rrbracket_{C^{(1+\beta)/2, 0}(Q)}, \llbracket \cdot \rrbracket_{C^{\beta/2,\beta}(Q)}$ are the usual H\"older-semi norms.

\begin{proposition}
            \label{imbedding-prop}
Let $\gamma{\in (-1,p-1)}$ and $p, q \in [1,\infty)$ so that $\beta = 1- \frac{d+\gamma_+}{p} - \frac{2}{q} >0$. Also, let ${\cD}$ be a non-empty open bounded set in $\overline{\mathcal{Q}_{1/2}^+}$. Then, there is $N = N(d, p, q, |\cD|, \gamma)>0$ so that the following assertion holds. If $u \in   \cW^{1,2}_{q,p}(\mathcal{Q}^+_1, x_d^\gamma dz)$, then
\begin{equation}
                        \label{eq9.24}
\|Du\|_{L_\infty(\mathcal{Q}^+_{1/2})}  \leq N  \Big[\|Du\|_{L_1(\mathcal{D})} + \|u_t\|_{L_{q,p}(\mathcal{Q}^+_1, x_d^\gamma dz)}  + \|D^2u\|_{L_{q,p}(\mathcal{Q}^+_1, x_d^\gamma dz)}  \Big]
\end{equation}
and
\begin{equation}
                        \label{eq9.25}
 |Du(t,x) - Du(s,y)| \leq N r^{\beta} \Big[ \|u_t\|_{L_{q,p}(\mathcal{Q}^+_1, x_d^{\gamma} dz)}  + \|D^2u\|_{L_{q,p}(\mathcal{Q}^+_1, x_d^{\gamma} dz)}  \Big],
\end{equation}
for every $(t,x), (s,y) \in \overline{\mathcal{Q}_{1/2}^+}$ and for $r = (|x-y|^2 + |t-s|)^{1/2}$. Moreover, we have
\begin{equation}
                        \label{eq9.25b}
 |u(t,x) - u(s,x)| \leq N |t-s|^{(1+\beta)/2} \Big[ \|u_t\|_{L_{q,p}(\mathcal{Q}^+_1, x_d^{\gamma} dz)}  + \|D^2u\|_{L_{q,p}(\mathcal{Q}^+_1, x_d^{\gamma} dz)}  \Big].
\end{equation}
\end{proposition}
\begin{proof}
We start with proving \eqref{eq9.25}. Let us denote $v= D_iu$ with some fixed $i =1, 2,\ldots, d$. {By the triangle inequality, w}e only need to prove the assertion with $r= (|x-y|^2 +|t-s|)^{1/2} \in (0,1/2)$ for $(t,x), (s,y) \in \overline{\mathcal{Q}_{1/2}^+}$. {Without loss of generality, we assume that $s\le t$.} Let {$(t_0, x_0)=((t+s)/2,(x+y)/2)+r e_d/2$, where $e_d=(0,\cdots,0,1)$, and $ \mathcal{Q} = \mathcal{Q}_{r/2}(t_0, x_0)\subset \cQ_{1}^+$}. Let $\psi \in C_0^\infty(\mathcal{Q})$ be a standard cut-off function satisfying
\begin{equation} \label{psi-hold-1}
0 \leq \psi \leq 2, \quad \|D\psi\|_{L_\infty} \leq \frac{N}{r}, \quad  \text{and} \quad \fint_{\mathcal{Q}} \psi(t,x)\, dtdx =1.
\end{equation}
Then,
\begin{align}   \notag
&v(t,x) - v(s,y) = \fint_{\mathcal{Q}}\big(v(t,x) - v(s,y)\big) \psi(\bar{t}, \bar{x}) \,d\bar{t}d\bar{x} \\
=\ &   \fint_{\mathcal{Q}}\big( v(t,x) - v(\bar{t}, \bar{x})\big) \psi(\bar{t}, \bar{x}) \,d\bar{t}d\bar{x}
+ \fint_{\mathcal{Q}} \big(v(\bar{t}, \bar{x}) - v(s,y) \big)\psi(\bar{t}, \bar{x}) \,d\bar{t}d\bar{x} \notag\\
 =: & \ I_1 + I_2.\label{hold-1-proof}
\end{align}
Next, we estimate the terms $I_1$ and $I_2$ on the right-hand side of \eqref{hold-1-proof}. By the fundamental theorem of calculus, we have
\[
\begin{split}
 v(t,x) - v(\bar{t},\bar{x})  & = \int_0^1\Big[2\theta v_t((1-\theta^2) t + \theta^2 \bar{t}, (1-\theta) x + \theta \bar{x})(t-\bar{t}) \\
 & \qquad + Dv((1-\theta^2) t + \theta^2 \bar{t}, (1-\theta) x + \theta \bar{x})\cdot (x-\bar{x}) \Big] d\theta.
\end{split}
\]
Then,  it follows from the Fubini theorem that
\begin{align*}
& I_1  : =   \fint_{\mathcal{Q}}\big ( v(t,x) - v(\bar{t}, \bar{x})\big)\psi(\bar{t}, \bar{x})\,d\bar{t}d\bar{x} \\
& = N r^{-(d+2)} \int_0^1\theta^{-1}  \left( \int_{\mathcal{Q}} \big(2 v_t(\tau,h)(t-\tau) + Dv(\tau, h)\cdot (x-h)\big) \psi(\bar{t},\bar{x})\, d\bar{t} d\bar{x} \right) d\theta,
\end{align*}
where we denote
\[
h = (1-\theta) x + \theta \bar{x} \quad \text{and} \quad \tau =(1- \theta^2) t + \theta^2 \bar{t}.
\]
We observe that
 \[
v_t (\tau, h) = D_i u_t(\tau, h).
\]
From this,  \eqref{psi-hold-1},  and by {the change of variables} $\bar{t} \mapsto \tau$ and $\bar{x} \mapsto h$, and the integration by parts for the term involving $v_t$, we infer that
\[
|I_1|  \leq N r^{-(d+{1})} \int_0^1 \theta^{-(d+2)}\left(\int_{\mathcal{Q}_{\theta r{/2}}{((1-\theta^2)t+\theta^2 t_0,(1-\theta)x+\theta x_0)}} \big(  |u_t| +  |D^2u|  \big) d\tau dh\right) d\theta.
\]
{By the convexity of $\cQ_1^+$, it is easily seen that $\mathcal{Q}_{\theta r/2}((1-\theta^2)t+\theta^2 t_0,(1-\theta)x+\theta x_0)\subset \cQ_1^+$}.
It then follows from H\"{o}lder's inequality that
\begin{align*}
|I_1| \leq Nr^{1-\frac{d+\gamma_+}{p} -\frac{2}{q}} \Big[ \|D^2u\|_{L_{q,p}({\cQ_1^+}, x_d^{\gamma} dz)} + \|u_t\|_{L_{q,p}(\cQ_1^+, x_d^{\gamma} dz)} \Big],
\end{align*}
where $N = N(d, p, q, \gamma)>0$, and we also used the fact that
\[
\left(\int_{B_{\theta r/2}({(1-\theta)x+\theta x_0})} |h_d|^{-\gamma/(p-1)}\,  dh \right)^{1-\frac{1}{p}} \leq N (\theta r)^{d-\frac{d+\gamma_+}{p}},
\]
for all $x \in \overline{B}_{1/2}$ and for all $\theta, r \in (0,1)$. Similarly, we also have
\begin{align*}
|I_2| & \leq Nr^{1-\frac{d+\gamma_+}{p}-\frac{2}{q}} \Big[ \|D^2u\|_{L_{q,p}({\cQ_1^+}, x_d^{\gamma} dz)} + \|u_t\|_{L_{q,p}({\cQ_1^+}, x_d^{\gamma} dz)} \Big].
\end{align*}
From the last two estimates, we infer from \eqref{hold-1-proof} that
\begin{align} \notag
|v(t,x) - v(s,y)| & \leq |I_1| + |I_2| \\ \label{hold-2-proof}
& \leq N r^{1-\frac{d+\gamma_+}{p}-\frac{2}{q}}  \Big[ \|D^2u\|_{L_{q,p}(\mathcal{Q}_{1}, x_d^{\gamma} dz)} + \|u_t\|_{L_{q,p}(\mathcal{Q}_{1}, x_d^{\gamma} dz)} \Big],
\end{align}
which gives \eqref{eq9.25} as $v= D_iu$ with $i \in \{1,2,\ldots, d\}$.

\smallskip
To prove \eqref{eq9.24}, we note from \eqref{hold-1-proof} that
\[
|v(t,x)| \leq  |v(s,y)|+ N \Big[ \|D^2u\|_{L_{q,p}(\mathcal{Q}_{1}, x_d^{\gamma} dz)} + \|u_t\|_{L_{q,p}(\mathcal{Q}_{1}, x_d^{\gamma} dz)} \Big],
\]
for every $(t,x), (s,y) \in {\overline{\mathcal{Q}^+_{1/2}}}$. Then, integrating this with respect to the $(s,y)$ variable on $\mathcal{D}$, we obtain \eqref{eq9.24}.

{Finally, by the triangle inequality,
\begin{align}
&|u(t,x) - u(s,x)|\notag\\
\le \ & |u(t,x) - u(t,y)+Du(t,x)\cdot (y-x)|+|u(s,x) - u(s,y)+Du(s,x)\cdot (y-x)|\notag\\
& +|u(t,y)-u(s,y)|+|Du(t,x)-Du(s,x)||y-x|=:J_1+J_2+J_3+J_4,    \label{eq11.14}
\end{align}
where $y\in B_{(t-s)^{1/2}}^+(x)$. It follows from \eqref{eq9.25} that
$J_1+J_2+J_4$ is bounded by the right-hand side of \eqref{eq9.25b}. Moreover, by the fundamental theorem of calculus and H\"older's inequality,
\begin{align*}
&\fint_{B_{(t-s)^{1/2}}^+(x)}|u(t,y)-u(s,y)|\,dy
\leq \int_s^t \fint_{B_{(t-s)^{1/2}}^+(x)} |u_t(\tau, y)|\,dyd\tau\\
& \le N(t-s)^{1-\frac 1 q-\frac{d+\gamma_+}{2p}} \|u_t\|_{L_{q,p}(\cQ_1^+, x_d^{\gamma} dz)}.
\end{align*}
Taking the average of \eqref{eq11.14} with respect to $y\in B_{(t-s)^{1/2}}^+(x)$ and using the above inequalities, we reach \eqref{eq9.25b}. The lemma is proved.}
\end{proof}

\section{Equations with coefficients depending only on the \texorpdfstring{$x_d$}{} variable} \label{sec:3}
In this section, we consider \eqref{eq:main} when the coefficients in \eqref{eq:main} only depend on the $x_d$ variable.
Let  us denote
\begin{equation} \label{L0-def}
\sL_0 u = \bar{a}_0(x_d) u_t+\lambda \bar{c}_0(x_d) u-\mu(x_d) \bar{a}_{ij}(x_d)D_iD_j u.
\end{equation}
where $\mu, \bar{a}_0, \bar{c}_0, \bar{a}_{ij}: \bR_+ \rightarrow \bR$ are given measurable functions and they satisfy \eqref{con:mu}-\eqref{con:ellipticity}. We consider
\begin{equation}\label{eq:xd}
\left\{
\begin{array}{cccl}
\sL_0 u & = & \mu(x_d) f \quad &\text{ in } \Omega_T,\\
u & = & 0 \quad &\text{ on } (-\infty, T) \times \partial \bR^d_+.
\end{array} \right.
\end{equation}
\smallskip
The main result of this section is the following theorem, which  is a special case of Corollary \ref{cor1}.
\begin{theorem}\label{thm:xd}
Assume that $\bar{a}_0, \bar{c}_0, (\bar{a}_{ij})$ satisfy \eqref{con:mu}--\eqref{con:ellipticity} and assume further that $f \in {L_p(\Omega_T,x_d^\gamma\, dz)}$ for some given $p>1$ and
\[
\gamma \in \big(p(\alpha-1)_+-1,2p-1\big).
\]
Then, \eqref{eq:xd} admits a unique strong solution $u \in \sW^{1,2}_p(\Omega_T, x_d^\gamma\, dz)$.
Moreover,
\begin{align} \notag
& \|\chi^{-\alpha}u_t\|_{L_p(\Omega_T,x_d^{\gamma}\, dz)} + \|D^2 u\|_{L_p(\Omega_T,x_d^{\gamma}\, dz)}  \\ \label{eq:xd-main}
& \quad \qquad +\lambda\|\chi^{-\alpha}u\|_{L_p(\Omega_T,x_d^{\gamma}\, dz)} \le N\|f\|_{L_p(\Omega_T,x_d^{\gamma}\, dz)};
\end{align}
and if $\gamma\in (\alpha p/2-1,2p-1)$, we also have
\begin{equation}
                    \label{eq3.09}
\lambda^{1/2}\|\chi^{-\alpha/2}Du\|_{L_p(\Omega_T,x_d^{\gamma}\, dz)}
\le N\|f\|_{L_p(\Omega_T,x_d^{\gamma}\, dz)},
\end{equation}
where $N=N(d,\nu,\alpha, \gamma, p)>0$.
\end{theorem}

The proof of Theorem \ref{thm:xd} requires various preliminary results and estimates.
Our starting point is Lemma \ref{l-p-sol-lem} below which gives Theorem \ref{thm:xd} when $\gamma$ is large. See Subsection \ref{subsec:L2} below.
Then, in Subsections \ref{subsec:boundary} and \ref{subsec:int}, we derive pointwise estimates for solutions to the corresponding homogeneous equations.
Afterwards, we derive the oscillation estimates for solutions  in Subsection \ref{subsec:osc-est}.
The proof of Theorem \ref{thm:xd} will be given in the last subsection, Subsection \ref{proof-xd}.

Before starting, let us point out several observations as well as recall several needed definitions. Note that by  dividing the PDE in \eqref{eq:xd} by $\bar{a}_0$ and then absorbing $\bar{a}_{dd}$ into $\mu(x_d)$, without loss of generality, we may assume that
\begin{equation} \label{add-cond}
\bar{a}_{dd}=1  \qquad \text{ and } \qquad \bar{a}_0 =1.
\end{equation}
Observe that \eqref{eq:xd} can be rewritten into a divergence form equation
\begin{equation}
                    \label{eq:xd-div}
\bar{a}_0 u_t+\lambda \bar{c}_0(x_d)u-\mu(x_d) D_i(\tilde a_{ij}(x_d) D_{j} u)=\mu(x_d)f \quad \text{ in } \Omega_T,
\end{equation}
where
\begin{equation}
                \label{eq:change}
\tilde a_{ij}=\left\{
                \begin{array}{ll}
                  \bar{a}_{ij}+ \bar{a}_{ji} & \hbox{for $i\neq d$ and $j=d$;} \\
                  0 & \hbox{for $i=d$ and $j\neq d$;} \\
                  \bar{a}_{ij} & \hbox{otherwise.}
                \end{array}
              \right.
\end{equation}
We note that even though $(\tilde a_{ij})$ is not symmetric, it still satisfies the ellipticity condition \eqref{con:ellipticity} and also $\tilde a_{dd} =1$ when \eqref{add-cond} holds.

\smallskip
Due to the divergence form as in \eqref{eq:xd-div}, we need the definition of its weak solutions.
In fact, sometimes in this section, we consider the following class of equations in divergence form which are slightly more general than \eqref{eq:xd-div}
\begin{equation} \label{eq:dx-loc}
u_t + \lambda \bar{c}_0(x_d) u - \mu(x_d)D_i (\tilde{a}_{ij}(x_d)D_j u - F_i) = \mu(x_d) f  \quad \text{in} \quad (S, T) \times \mathcal{D}
\end{equation}
with the boundary condition
\begin{equation*}
u  =  0  \quad \text{on} \quad  (S, T) \times (\overline{\mathcal{D}} \cap \{x_d =0\})
\end{equation*}
for some open set $\mathcal{D} \subset \mathbb{R}^d_+$ and $-\infty \leq S < T \leq \infty$.

\begin{definition}
For a given weight $\omega$ defined on $(S, T) \times \mathcal{D} $ and for given $F= (F_1, F_2, \ldots, F_2) \in L_{p, \text{loc}}((S, T) \times \mathcal{D} )^d$ and $f \in L_{p, \text{loc}}((S, T) \times \mathcal{D} )$, we say that a function $u \in \cH_p^1((S, T) \times \mathcal{D} , \omega)$ is a weak solution of \eqref{eq:dx-loc} if
\begin{equation} \label{def-local-weak-sol}
\begin{split}
& \int_{(S, T) \times \mathcal{D} }\mu(x_d)^{-1}(-u \partial_t \varphi + \lambda \overline{c}_0 u \varphi)dz + \int_{(S, T) \times \mathcal{D} } (\tilde{a}_{ij} D_ju  - F_i) D_i \varphi  dz \\
& = \int_{(S, T) \times \mathcal{D} }  f(z) \varphi(z) dz, \quad \forall \ \varphi \in C_0^\infty((S, T) \times \mathcal{D} ).
\end{split}
\end{equation}
\end{definition}

\subsection{\texorpdfstring{$L_p$}{} strong solutions when the powers of weights are large} \label{subsec:L2}
The following lemma is the main result of this subsection, which gives Theorem  \ref{thm:xd} when $\gamma \in (p-1, 2p-1)$.
\begin{lemma} \label{l-p-sol-lem} Let $\nu \in (0,1)$, $\lambda>0$, $\alpha \in (0,2)$, $p \in (1, \infty)$, and $\gamma \in (p-1, 2p-1)$. Assume that $\bar{a}_0, \bar{c}_0, (\bar{a}_{ij})$, and $\mu$ satisfy the ellipticity and boundedness conditions \eqref{con:mu}--\eqref{con:ellipticity}. Then, for any $f \in L_p(\Omega_T, x_d^\gamma\, dz)$, there exists a unique strong solution $u \in \sW^{1,2}_p(\Omega_T, x_d^\gamma\, dz)$ to \eqref{eq:xd}. Moreover, for every solution $u \in \sW^{1,2}_p(\Omega_T, x_d^\gamma\, dz)$ of \eqref{eq:xd} with $f \in L_p(\Omega_T, x_d^\gamma\, dz)$, it holds that
\begin{align} \notag
& \lambda \|\chi^{-\alpha} u\|_{L_p(\Omega_T, x_d^{\gamma}dz)} + \sqrt{\lambda} \|\chi^{-\alpha/2}Du \|_{L_p(\Omega_T, x_d^{\gamma}dz)} \\ \label{est-0405-1}
& \qquad + \| D^2u\|_{L_p(\Omega_T, x_d^{\gamma}dz)}  + \|\chi^{-\alpha} u_t\|_{L_p(\Omega_T, x_d^{\gamma}dz)} \leq N  \|f\|_{L_p(\Omega_T, x_d^{\gamma}dz)},
\end{align}
where $N = N(d, \alpha, \nu, \gamma, p)>0$.
\end{lemma}
\begin{proof}  The key idea is to apply \cite[Theorem 2.4]{DPT21} to the divergence form equation \eqref{eq:xd-div}, and then use an idea introduced by Krylov in \cite[Lemma 2.2]{Kr99} with a suitable scaling.
To this end, we assume that \eqref{add-cond} holds, and let us denote $\gamma' = \gamma -p \in (-1, p-1)$ and we observe that
\[
x_d^{1-\alpha}\mu(x_d) |f(z)| \sim x_d |f(z)| \in L_p(\Omega_T, x_d^{\gamma'}dz).
\]
As $\gamma' \in (-1, p-1)$, we have $x_d^{\gamma'} \in A_p$.
Moreover, the equation \eqref{eq:xd} can be written in divergence form as \eqref{eq:xd-div}.
Therefore, we apply \cite[Theorem 2.4]{DPT21} to \eqref{eq:xd-div} with $f_1=\mu(x_d) f$ and $f_2= 0$ to yield the existence of   a unique weak solution $u \in \mathscr{H}^1_p(\Omega_T, x_d^{\gamma'}dz)$ of \eqref{eq:xd-div} satisfying
\begin{align} \notag
& \|Du\|_{L_{p}(\Omega_T,x_d^{\gamma'} dz)} + \sqrt{\lambda} \|\chi^{-\alpha/2}u\|_{L_{p}(\Omega_T,x_d^{\gamma'}dz)} \\ \label{Du-0226}
& \le N\|{x_d^{1-\alpha}f_1}\|_{L_{p}(\Omega_T,x_d^{\gamma'}dz)} = N \|f\|_{L_p(\Omega_T, x_d^\gamma dz)},
\end{align}
with $N = N(d, \nu, \alpha, \gamma, p)>0$. We note here that because the coefficients $\bar{c}_0, \bar{a}_{ij}$ only depend on $x_d$, \cite[Theorem 2.4]{DPT21} holds for any $\lambda>0$ by a scaling argument.  From \eqref{Du-0226}, the zero boundary condition,  and the weighted Hardy inequality (see \cite[Lemma 3.1]{DP-AMS} for example), we infer that
\begin{align} \notag
\|u\|_{L_p(\Omega_T, x_d^{\gamma -2p}dz)} & = \|\chi^{-1}u\|_{L_p(\Omega_T, x_d^{\gamma'}dz)} \leq N \|Du\|_{L_{p}(\Omega_T,x_d^{\gamma'})} \\ \label{u-op-wei-0226}
& \leq N \|f\|_{L_p(\Omega_T, x_d^\gamma)}.
\end{align}

\smallskip
It remains to prove that \eqref{est-0405-1} holds as it also implies that $u \in \sW^{1,2}_p(\Omega_T, x_d^\gamma)$.
We apply the idea introduced by Krylov in \cite[Lemma 2.2]{Kr99} and combine it with a scaling argument to remove the degeneracy of the coefficients. See also \cite[Theorem 3.5]{DK15} and \cite[Lemma 4.6]{DP-JFA}.
To this end, let us fix a standard non-negative cut-off function $\zeta \in C_0^\infty((1,2))$. For each $r >0$, let $\zeta_r(s) =\zeta(rs)$ for $s \in \bR_+$. Note that with a suitable assumption on the integrability of a given function $v: \Omega_T \rightarrow \bR$ and for $\beta \in \bR$, by using the substitution  $r^{\alpha}t \mapsto s$ for the integration with respect to the time variable, and then using the Fubini theorem, we have
\begin{equation} \label{weight-kry}
\begin{split}
& \int_0^\infty \left(\int_{\Omega_{r^{-\alpha}T}}|\zeta_r(x_d) v_r(z)|^p\, dz\right) r^{-\beta-1}\, dr =N_1\int_{\Omega_T} |v(z)|^p x_d^{\beta+\alpha}\, dz, \\
& \int_0^\infty \left(\int_{\Omega_{r^{-\alpha} T}}|\zeta_r'(x_d) v_r(z)|^p\,  dz\right) r^{-\beta-1}\, dr =N_2\int_{\Omega_T} |v(z)|^p x_d^{\beta + \alpha -p}\, dz, \\
& \int_0^\infty \left(\int_{\Omega_{r^{-\alpha} T}}|\zeta_r''(x_d) v_r(z)|^p\,  dz\right) r^{-\beta-1}\, dr =N_3\int_{\Omega_T} |v(z)|^p x_d^{\beta + \alpha-2p}\, dz,
\end{split}
\end{equation}
where $v_r(z) = v(r^{\alpha}t, x)$ for $z = (t, x) \in \Omega_{r^{-\alpha}T}$,
 \[
N_1 = \int_0^\infty |\zeta (s)|^p s^{-\beta-\alpha-1} ds, \quad
N_2 = \int_0^\infty |\zeta'(s)|^p s^{p-\beta-\alpha-1} ds,
\]
and
\[
N_3 = \int_0^\infty |\zeta''(s)|^p s^{2p-\beta-\alpha  -1} ds.
\]
Next, for $r>0$, we denote $u_r(z) = u(r^{\alpha}t,x)$,
\[
\hat{a}_{ij} (x_d)= r^{\alpha}\mu(x_d) \bar{a}_{ij}(x_d), \quad \bar{\lambda} = \lambda r^{\alpha}, \quad  \text{and} \quad f_r (z) = r^{\alpha} \mu(x_d) f(r^{\alpha}t, x).
\]
Note that $u_r$ solves the equation
\[
\partial_t u_r + \bar{\lambda} \bar{c}_0 u_r - \hat{a}_{ij}(x_d) D_i D_j u_r =  f_r \quad \text{in} \quad \Omega_{r^{-\alpha}T}.
\]
Let $w(z) =\zeta_r(x_d) u_r(z)$, which satisfies
\begin{equation} \label{w-eqn-0405-1}
w_t + \bar{\lambda} \bar{c}_0(x_d) w -  \hat{a}_{ij}(x_d) D_i D_j w = \hat{g} \quad \text{in} \quad \Omega_{r^{-\alpha}T}
\end{equation}
with the boundary condition $w(z', 0) =0$ for $z' \in (-\infty, r^{-\alpha}T)\times \bR^{d-1}$, where \[
\begin{split}
\hat{g}(z) & = \zeta_r  f_r(z) - \hat{a}_{dd} \zeta''_r  u_r  - \sum_{i \neq d} \big(\hat{a}_{id}  + \hat{a}_{di}\big)\zeta '_rD_i u_r.
\end{split}
\]
We note that $\text{supp}(w) \subset (-\infty, r^{-\alpha}T) \times \bR^{d-1} \times(1/r, 2/r)$, and on this set the coefficient matrix $(\hat{a}_{ij})$ is uniformly elliptic and bounded as $r^{\alpha}\mu (x_d) \sim 1$ due to \eqref{con:mu}.

We now prove \eqref{est-0405-1}  with the extra assumption that $u \in  \sW^{1,2}_p(\Omega_T, x_d^{\gamma'}dz)$.  Under this assumption and as $\zeta_r$ is compactly supported in $(0, \infty)$, we see that $w \in W^{1,2}_p(\Omega_{r^{\alpha}T})$, the usual parabolic Sobolev space. Then by applying the  $W^{1, 2}_{p}$-estimate for the uniformly elliptic and bounded coefficient equation \eqref{w-eqn-0405-1} (see, for instance, \cite{D12}),  we obtain
\[
  \bar{\lambda} \|w\| +  \bar{\lambda}^{1/2}\, \|Dw\|  +  \|D^2 w\|  + \|w_t\| \leq N \|\hat{g}\|,
\]
where $\|\cdot \| = \| \cdot \|_{L_p(\Omega_{r^{-\alpha}T})}$ and $N = N(d, \nu, p)>0$. From this, the definition of $\hat{g}$, and a simple manipulation, we obtain
\[
\begin{split}
&   \lambda r^{\alpha} \|\zeta_r u_r \| + \sqrt{\lambda}  r^{\alpha/2} \|\zeta_r Du_r\|   + \|\zeta_r D^2u_r\|  +\|\zeta_r \partial_t u_r \|  \\
& \leq N\Big[ \|\zeta_r f_r\| + \sqrt{\lambda} r^{\alpha/2}   \|\zeta_r' u_r\|  +  \|\zeta''_r u_r\|   + \|\zeta_r' Du_r\| \Big].
\end{split}
\]
Now, we raise this last estimate to the power $p$, multiply both sides by $r^{-(\gamma-\alpha)-1}$, integrate the result with respect to $r$ on $(0,\infty)$, and then apply \eqref{weight-kry} to obtain
\[
\begin{split}
& \lambda \|\chi^{-\alpha} u\|_{L_p(\Omega_T, x_d^{\gamma}\, dz)} +\sqrt{\lambda} \|\chi^{-\alpha/2}Du \|_{L_p(\Omega_T, x_d^{\gamma}\, dz)} \\
& \qquad + \| D^2u\|_{L_p(\Omega_T, x_d^{\gamma}\, dz)}  + \|\chi^{-\alpha} u_t\|_{L_p(\Omega_T, x_d^{\gamma}\, dz)}\\
 & \leq N \Big[ \|f\|_{L_p(\Omega_T, x_d^{\gamma}\, dz)} +\sqrt{\lambda} \|\chi^{-\alpha/2}u\|_{L_p(\Omega_T, x_d^{\gamma -p}\, dz)} + \|u\|_{L_p(\Omega_T, x_d^{\gamma -2p}\, dz)}  \\
 & \qquad   + \|Du\|_{L_p(\Omega_T, x_d^{\gamma -p}\, dz)}\Big].
 \end{split}
\]
From the last estimate, \eqref{Du-0226}, \eqref{u-op-wei-0226}, and the fact that $\gamma' = \gamma-p$, we infer that
\[
\begin{split}
& \lambda \|\chi^{-\alpha} u\|_{L_p(\Omega_T, x_d^{\gamma}\, dz)} + \sqrt{\lambda} \|\chi^{-\alpha/2}Du \|_{L_p(\Omega_T, x_d^{\gamma}\,dz)} \\
& \qquad + \| D^2u\|_{L_p(\Omega_T, x_d^{\gamma}\,dz)}  + \|\chi^{-\alpha} u_t\|_{L_p(\Omega_T, x_d^{\gamma}\,dz)} \leq N  \|f\|_{L_p(\Omega_T, x_d^{\gamma}\, dz)}.
\end{split}
\]
This proves \eqref{est-0405-1} under the additional assumption that $u \in \sW^{1,2}_p(\Omega_T, x_d^{\gamma'}\,dz)$.

\smallskip
It remains to remove the extra assumption that $u \in \sW^{1,2}_p(\Omega_T, x_d^{\gamma'}dz)$. By mollifying the equation \eqref{eq:xd} in $t$ and $x'$ and applying  \cite[Theorem 2.4]{DPT21} to the equations of $u^{(\varepsilon)}_t$ and $D_{x'}u^{(\varepsilon)}$, we obtain
$$
\chi^{-\alpha}u^{(\varepsilon)}, \chi^{-\alpha} u_t^{(\varepsilon)},  DD_{x'}u^{(\varepsilon)} \in L_p(\Omega_T, x_d^{\gamma'} dz).
$$
This and the PDE in \eqref{eq:xd} for $u^{(\varepsilon)}$ imply that
\[
D_{dd} u^{(\varepsilon)}  \in L_p(\Omega_T, x_d^{\gamma'} dz).
\]
Therefore $u^{(\varepsilon)} \in \sW^{1,2}_p(\Omega_T, x_d^{\gamma'}dz)$ is a strong solution of \eqref{eq:xd} with $f^{(\varepsilon)}$ in place of $f$. From this, we apply the a priori estimate \eqref{est-0405-1} that we just proved for $u^{(\varepsilon)}$ and pass to the limit as $\varepsilon \rightarrow 0^+$ to obtain the estimate \eqref{est-0405-1} for $u$. The proof of the lemma is completed.
\end{proof}
\subsection{Boundary H\"older estimates for homogeneous equations} \label{subsec:boundary}
Recall the operator $\sL_0$ defined in \eqref{L0-def}. In this subsection, we consider the homogeneous equation
\begin{equation}
            \label{eq:hom}
 \begin{cases}
\sL_0u=0 \quad &\text{ in } Q_1^+,\\
u=0 \quad &\text{ on } Q_1\cap \{x_d=0\}.
\end{cases}
\end{equation}
As the discussion that leads to \eqref{eq:xd-div}, without loss of generality we assume \eqref{add-cond} so that  \eqref{eq:hom} can be written in divergence form as
\begin{equation}
                    \label{eq:hom-div}
                    \left\{
                    \begin{array}{cccl}
u_t+\lambda \bar{c}_0(x_d)u-\mu(x_d) D_i(\tilde a_{ij}(x_d) D_{j} u) & = &0 &  \quad \text{ in } Q_1^+,\\
u & = & 0 & \quad \text{on} \quad Q_1\cap \{x_d=0\}.
\end{array} \right.
\end{equation}
A function $u \in \cH_{p}^1(Q_1^+)$ with $p \in (1,\infty)$ is said to be a weak solution of \eqref{eq:hom} if it is a weak solution of \eqref{eq:hom-div} in the sense defined in \eqref{def-local-weak-sol} and $u=0$ on $Q_1\cap \{x_d=0\}$ in the sense of trace.

\smallskip
For each $\beta \in (0,1)$, the $\beta$-H\"older semi-norm in  the spatial variable of a function $u$ on an open set $Q\subset \bR^{d+1}$ is given by
\[
\llbracket u\rrbracket_{C^{0, \beta}(Q)} = \sup\left\{  \frac{|u(t,x) - u(t,y)|}{|x-y|^{\beta}}:  x \not =y, \  (t,x), (t,y) \in Q \right\}.
\]
For $k, l \in \mathbb{N} \cup \{0\}$,  we denote
\[
\|u\|_{C^{k, l}(Q)}  = \sum_{i=0}^k \sum_{|j| \leq l}\|\partial_t^i D_{x}^j u\|_{L_\infty(Q)} .
\]
We also use the following H\"{o}lder norm of $u$ on $Q$
\[
\|u\|_{C^{k, \beta}(Q)} = \|u\|_{C^{k,0}(Q)} + \sum_{i=0}^{k} \llbracket \partial_t^i u\rrbracket_{C^{0, \beta}(Q)}.
\]
We begin with the following Caccioppoli type estimate.
\begin{lemma}  \label{caccio}
Suppose that $u\in \cH^1_2(Q_1^+)$ is a weak solution of \eqref{eq:hom}.
Then, for any integers $k,j\ge 0$ and $l=0,1$,
\begin{equation}
                \label{eq:hom-b1}
\int_{Q_{1/2}^+}|\partial_t^k D_{x'}^j D_{d}^l u|^2 \,dz
\le N \int_{Q_{1}^+} u^2 \,dz
\end{equation}
where $N = N(d,\nu,\alpha,k,j,l)>0$.
\end{lemma}
\begin{proof} Again, we can assume \eqref{add-cond} holds. The estimate \eqref{eq:hom-b1} follows from \cite[(4.12)]{DPT21} applied to \eqref{eq:hom-div}.
\end{proof}
\begin{lemma}
                        \label{lem:boundary}
Let $p_0 \in (1, \infty)$ and suppose that $u\in \cH^1_{p_0}(Q_1^+)$ is a weak solution of \eqref{eq:hom}.
Then,
\begin{equation}
                        \label{eq:hom-b2}
                        \begin{split}
& \|u\|_{C^{1,1}(Q_{1/2}^+)}+\|D_{x'}u\|_{C^{1,1}(Q_{1/2}^+)}+\|D_d u\|_{C^{1,\delta_0}(Q_{1/2}^+)} \\
& + \sqrt{\lambda}\|\chi^{-\alpha/2}u\|_{C^{1,1-\alpha/2}(Q_{1/2}^+)} \le N \|Du\|_{L_{p_0}(Q_1^+)},
\end{split}
\end{equation}
where $N=N(d,\nu,\alpha, p_0)>0$ and $\delta_0=\min\{2-\alpha,1\}$.
\end{lemma}
\begin{proof}  As explained, we can assume that \eqref{add-cond} holds. We apply  \cite[Lemma 5.5]{DPT21} to \eqref{eq:hom-div} by noting that $U:=\tilde a_{dj}D_j u=D_du$ in view of \eqref{add-cond} and \eqref{eq:change}.
\end{proof}

\begin{lemma}
                    \label{prop:boundary}
Let $p_0 \in (1, \infty)$, $\beta_0\in {(-\infty}, \min\{1,\alpha\}]$, and $\alpha_0 > -1$ be fixed constants.
There exists a number $\beta_1=\beta_1(\alpha,\beta_0) \in (0,1]$ such that for every weak solution $u\in \cH^1_{p_0}(Q_1^+)$ to \eqref{eq:hom}, it holds that
\begin{align}
        \label{eq:hom-b3}
\|\chi^{-\beta_0} u\|_{C^{1,\beta_1}(Q_{1/2}^+)}&\le N\|\chi^{-\beta_0} u\|_{L_{p_0}(Q_{3/4}^+,x_d^{\alpha_0}\,dz)},\\
        \label{eq:hom-b4}
\|\chi^{-\beta_0} u_t\|_{C^{1,\beta_1}(Q_{1/2}^+)}&\le N\|\chi^{-\beta_0} u_t\|_{L_{p_0}(Q_{3/4}^+,x_d^{\alpha_0}\,dz)},\\
        \label{eq:hom-b5}
\|\chi^{\alpha-\beta_0} DD_{x'}u\|_{C^{1,\beta_1}(Q_{1/2}^+)}&\le N\|\chi^{\alpha-\beta_0} DD_{x'}u\|_{L_{p_0}(Q_{3/4}^+,x_d^{\alpha_0}\,dz)},
\end{align}
and
\begin{equation}
        \label{eq:hom-b-Du}
\|\chi^{\beta_0}Du\|_{C^{1,\beta_1}(Q_{1/2}^+)} \le N\|\chi^{\beta_0}Du\|_{L_{p_0}(Q_{3/4}^+,x_d^{\alpha_0}dz)},
\end{equation}
where $N=N(d,\nu,\alpha,\alpha_0,\beta_0, p_0)>0$.
\end{lemma}
\begin{proof}
Again, we assume \eqref{add-cond}.
Note that once the lemma with $\alpha_0 \geq 0$ is proved, the case when $\alpha_0 \in (-1, 0)$ will follow immediately.
Hence, we only need to prove the lemma with the assumption that $\alpha_0 \geq 0$.  We first assume $p_0 =2$.
Since $\beta_0\le \min\{1,\alpha\}$, by \eqref{eq:hom-b1} and the boundary Poincar\'e inequality, the right-hand sides of \eqref{eq:hom-b3}, \eqref{eq:hom-b4}, and \eqref{eq:hom-b5} are all finite.
We consider two cases.

\medskip
\noindent
{\em Case 1: $\beta_0=0$.} When $\alpha_0=0$, \eqref{eq:hom-b3} and \eqref{eq:hom-b-Du} follow from \eqref{eq:hom-b2} and \eqref{eq:hom-b1}. For general $\alpha_0\ge 0$, by   \eqref{eq:hom-b3} with $\beta_0=0$ and $\alpha_0=0$ and H\"older's inequality, we have
\begin{align*}
\|u\|_{L_\infty(Q_{1/2}^+)}&  \le N \|u\|_{L_2(Q_{2/3}^+)} \leq N\|u\|^{2\alpha_0/(1+2\alpha_0)}_{L_2(Q_{2/3}^+,x_d^{-1/2}\,dz)}
\|u\|^{1/(1+2\alpha_0)}_{L_2(Q_{2/3}^+,x_d^{\alpha_0}\,dz)}\\
&\le N\|u\|^{2\alpha_0/(1+2\alpha_0)}_{L_\infty(Q_{2/3}^+)}
\|u\|^{1/(1+2\alpha_0)}_{L_2(Q_{2/3}^+,x_d^{\alpha_0}\,dz)} \\
&  \leq \frac{1}{2} \|u\|_{L_\infty(Q_{2/3}^+)} + N\|u\|_{L_2(Q_{2/3}^+,x_d^{\alpha_0}\,dz)},
\end{align*}
where $N = N(d, \nu, \alpha, \alpha_0)>0$. From this and the standard iteration argument (see \cite[p. 75]{HanLin} for example), we obtain
\begin{equation} \label{est.alpha-zero}
\|u\|_{L_\infty(Q_{1/2}^+)} \leq N \|u\|_{L_2(Q_{3/4}^+,x_d^{\alpha_0}\,dz)}.
\end{equation}
The above, together with Lemma \ref{caccio}, yields
\begin{equation}
                \label{eq:hom-b6}
\int_{Q_{1/2}^+}|\partial_t^k D_{x'}^j D_{d}^l u|^2 \,dz
\le N(d,\nu,\alpha, \alpha_0,k,j,l) \int_{Q_{3/4}^+} u^2 x_d^{\alpha_0}\,dz
\end{equation}
for any integers $k,j\ge 0$ and $l=0,1$.
Using this last estimate, \eqref{eq:hom-b2}, and by suitably adjusting the sizes of the cylinders, we obtain  \eqref{eq:hom-b3} with $\beta_1 = 1$.
Similar to \eqref{est.alpha-zero}, we have
\[
\|Du\|_{L_\infty(Q_{1/2}^+)} \leq N \|Du\|_{L_2(Q_{3/4}^+,x_d^{\alpha_0}\,dz)}.
\]
From this, \eqref{eq:hom-b2}, and by shrinking the cylinders, we obtain
\[
\|Du\|_{C^{1,\delta_0}(Q_{1/2}^+)} \leq N \|Du\|_{L_2(Q_{3/4}^+,x_d^{\alpha_0}dz)}, \quad \text{where} \,\,  \delta_0=\min\{2-\alpha,1\},
\]
which is \eqref{eq:hom-b-Du} when $\beta_0 =0$.

Since $u_t$ and $D_{x'} u$ satisfy the same equation with the same boundary condition, similarly we also obtain \eqref{eq:hom-b4} as well as
\begin{equation}
        \label{eq:hom-b7}
\|DD_{x'}u\|_{C^{1,\delta_0}(Q_{1/2}^+)}\le N\|DD_{x'}u\|_{L_2(Q_{2/3}^+)},
\end{equation}
by Lemma \ref{lem:boundary}.
This together with \eqref{eq:hom-b6} implies \eqref{eq:hom-b5} with
$$\beta_1=\min\{\delta_0,\alpha \} = \min\{\alpha, 2-\alpha,1\}.$$

\medskip
\noindent
{\em Case 2: $\beta_0\neq 0$.}
We first prove \eqref{eq:hom-b5}.  By \eqref{eq:hom-b7} and by using the iteration argument as in \eqref{est.alpha-zero}, we have
\[
\|DD_{x'} u\|_{L_\infty(Q_{1/2}^+)} \leq N\|DD_{x'} u\|_{L_2(Q_{3/4}^+, x_d^{\alpha_0}\, dz)},
\]
where $N = N(d, \nu, \alpha, \alpha_0)>0$. Then, it follows from \eqref{eq:hom-b7}  that
\begin{equation} \label{eq:hom-b7-bis}
\|DD_{x'}u\|_{C^{1,\delta_0}(Q_{1/2}^+)}\le N\|DD_{x'}u\|_{L_2(Q_{3/4}^+, x_d^{\alpha_0}\, dz)}.
\end{equation}
Therefore, if $\beta_0=\alpha$, \eqref{eq:hom-b5} with $\beta_1=\delta_0$ follows from \eqref{eq:hom-b7-bis}. If $\beta_0<\alpha$,  it follows from \eqref{eq:hom-b7-bis} that
\[
\|DD_{x'}u\|_{C^{1,\delta_0}(Q_{1/2}^+)}\le N\|\chi^{{\alpha-\beta_0}}DD_{x'}u\|_{L_2(Q_{3/4}^+, x_d^{\alpha_0}\, dz)}
\]
where $N = N(d, \nu, \alpha, \beta_0, \alpha_0)>0$.
Then we also have \eqref{eq:hom-b5} with
\[ \beta_1=\min\{\delta_0,\alpha-\beta_0\} = \min\{2-\alpha,1,\alpha-\beta_0\}. \]
Similarly, \eqref{eq:hom-b-Du} can be deduced from \eqref{eq:hom-b-Du} with $\beta_0 =0$ by taking $\beta_1 = \min\{\delta_0, \beta_0\}$. Hence, both \eqref{eq:hom-b5} and \eqref{eq:hom-b-Du} hold with
\[
\beta_1=\min\{\delta_0,\alpha-\beta_0,{\beta_0}\} = \min\{2-\alpha,1,\alpha-\beta_0,   \beta_0\}.
\]

Next we show \eqref{eq:hom-b3}.
Since $\beta_0\le 1$, using the zero boundary condition, \eqref{eq:hom-b2}, and \eqref{eq:hom-b6}, we get
\begin{equation}
                        \label{eq:hom-b8}
\|\chi^{-\beta_0} u\|_{L_\infty(Q_{1/2}^+)}
\le N \|D_d u\|_{L_\infty(Q_{1/2}^+)}\le N\|u\|_{L_2(Q_{3/4}^+,x_d^{\alpha_0}\,dz)}.
\end{equation}
Since $u_t$ and $D_{x'} u$ satisfy the same equation and the same boundary condition, we have
\begin{align}
                        \label{eq:hom-b9}
&\|\chi^{-\beta_0} u_t\|_{L_\infty(Q_{1/2}^+)}
+\|\chi^{-\beta_0}D_{x'} u\|_{L_\infty(Q_{1/2}^+)}\notag\\
&\le \  N \|D_d u_t\|_{L_\infty(Q_{1/2}^+)}+N \|D_dD_{x'} u\|_{L_\infty(Q_{1/2}^+)}\notag\\
 &\le \ N \|D u\|_{L_2(Q_{2/3}^+,x_d^{\alpha_0})}\le N\|u\|_{L_2(Q_{3/4}^+,x_d^{\alpha_0}\,dz)},
\end{align}
where we used \eqref{eq:hom-b2}.
To estimate the H\"older semi-norm of $\chi^{-\beta_0} u$ in $x_d$, we write
$$
x_d^{-\beta_0} u(t,x)=x_d^{1-\beta_0}\int_0^1 (D_d u)(t,x',sx_d)\,ds
$$
and use \eqref{eq:hom-b2} and \eqref{eq:hom-b6}. Then we see that
\[
\llbracket \chi^{-\beta_0} u \rrbracket_{C^{0, \beta_1}(Q_{1/2}^+)} +
\llbracket \chi^{-\beta_0} \partial_t u \rrbracket_{C^{0, \beta_1}(Q_{1/2}^+)} \leq N \|\chi^{-\beta_0} u\|_{L_2(Q_{3/4}^+, x_d^{\alpha_0}\,dz)}
\]
where  $\beta_1 = \min\{\delta_0, 1-\beta_0\}$. Combining this with \eqref{eq:hom-b8} and \eqref{eq:hom-b9}, we reach \eqref{eq:hom-b3}.

Note that $u_t$ satisfies the same equation and the same boundary condition, we deduce \eqref{eq:hom-b4} from \eqref{eq:hom-b3}.  The proof  of the lemma when $p_0 =2$ is completed.

Next, we observe that when $p_0 >2$, the estimates \eqref{eq:hom-b3}--\eqref{eq:hom-b-Du} can be derived directly from the case $p_0 =2$ that we just proved using H\"{o}lder's inequality. On the other hand, when $p_0 \in (1, 2)$, it follows from Lemma \ref{lem:boundary} that $u \in \cH_2^1(Q_{3/4}^+)$. Then, by shrinking the cylinders, we apply the assertion when $p_0=2$ that we just proved and an iteration argument as in the proof of \eqref{est.alpha-zero} to obtain the estimates \eqref{eq:hom-b3}--\eqref{eq:hom-b-Du}.
\end{proof}
\begin{remark} The number $\beta_1$ defined in Lemma \ref{prop:boundary} can be found explicitly.
However, we do not need its explicit formula in the paper.
\end{remark}
\subsection{Interior H\"older estimates for homogeneous equations}  \label{subsec:int}
Fix a point $z_0 = (t_0, x_0) \in \Omega_T$, where $x_0 = (x_0', x_{0d}) \in \bR^{d-1} \times \bR_+$.
For $0<\rho < x_{0d}$ and $\beta \in (0,1)$, we  define the  weighted $\beta$-H\"{o}lder semi-norm of a function $u$ on $Q_\rho(z_0)$ by
\[
\begin{split}
\llbracket u\rrbracket_{C^{\beta/2, \beta}_{\alpha}(Q_\rho(z_0))} & = \sup \Big\{  \frac{|u(s,x) - u(t, y)|}{\big(x_{0d}^{-\alpha/2}|x-y| + |t-s|^{1/2}\big)^{\beta}}: (s,x) \not=(t,y) \\
& \qquad \qquad \qquad \text{and } (s,x), (t,y) \in Q_\rho(z_0)  \Big\}.
\end{split}
\]
As usual, we denote the corresponding weighted norm by
\[
\|u\|_{C^{\beta/2,  \beta}_{\alpha}(Q_\rho(z_0))} = \|u\|_{L_\infty(Q_\rho(z_0))} + \llbracket u\rrbracket_{C^{\beta/2, \beta}_{\alpha}(Q_\rho(z_0))}.
\]
The following result is the interior H\"older estimates of solutions to the homogeneous equation \eqref{eq:hom-div}.

\begin{lemma}\label{prop:int}
Let $z_0 = (t_0, x_0) \in \Omega_T$,  where $x_0 = (x_0', x_{0d}) \in \bR^{d-1} \times \bR_+$, and $\rho \in (0, x_{0d}/4)$.
Let $u \in \sW_{p_0}^{1,2}(Q_{2\rho}(z_0))$ be a strong solution of
\begin{equation*}
\sL_0 u=0  \quad \text{ in } Q_{2\rho}(z_0)
\end{equation*}
with some $p_0 \in (1, \infty)$. Then  for any $\beta \in \bR$,
\begin{align*}
&  \|\chi^{\beta} u\|_{L_\infty(Q_{\rho}(z_0))} +   \rho ^{(1-\alpha/2)/2} \llbracket \chi^{\beta} u\rrbracket_{C^{1/4, 1/2}_\alpha(Q_{\rho}(z_0))} \\
    &\leq  N \left(\fint_{ Q_{2\rho}( z_0)} |x_{d}^{\beta}u|^{p_0} \mu_0(dz)\right)^{1/p_0},
\end{align*}
and
\begin{align*}
& \|\chi^{\beta}Du\|_{L_\infty(Q_{\rho}(z_0))}
 + \rho^{(1-\alpha/2)/2} \llbracket {\chi^{\beta} Du}\rrbracket_{C^{1/4, 1/2}_\alpha(Q_{\rho}(z_0))} \\
   &\leq N \left(\fint_{ Q_{2\rho}( z_0)} |x_d^{\beta}Du|^{p_0} \mu_0(dz) \right)^{1/p_0},
\end{align*}
where $\mu_0(dz) = x_d^{\alpha_0}dtdx$ with some $\alpha_0 > -1$, and $N = N(\nu, d,\alpha, \beta, \alpha_0)>0$.
\end{lemma}
\begin{proof}  As in the proof of Lemma \ref{prop:boundary}, we may assume that $p_0 =2$. Without loss of generality, we assume that $x_{0d}=1$. Note that when $\beta = 0$, the assertions follow directly from \cite[Proposition 4.6]{DPT21}. In general, the assertions follow from the case when $\beta =0$ and the fact that
\[
\left(\fint_{ Q_{2\rho}( z_0)} |\chi^{\beta}f(z)|^{p_0} \mu_0(dz) \right)^{1/p_0} \approx \left(\fint_{ Q_{2\rho}( z_0)} |f(z)|^{p_0}  dz \right)^{1/p_0}.
\]
The lemma is proved.
\end{proof}
\subsection{Mean oscillation estimates} \label{subsec:osc-est} In this subsection, we apply Lemmas \ref{prop:boundary} and \ref{prop:int} to derive the mean oscillation estimates of
$$
U = (\chi^{-\beta_0} u_t,  \chi^{\alpha-\beta_0}DD_{x'} u, \lambda \chi^{-\beta_0}u) \qquad \text{and} \qquad Du
$$
respectively with the underlying measure
\begin{equation} \label{mu-1-def}
\mu_1(dz) = x_d^{\gamma_1}\,dx dt \qquad \text{and} \qquad \bar{\mu}_1(dz) = x_d^{\bar{\gamma}_1} dxdt,
\end{equation}
where $u$ is a strong solution of \eqref{eq:xd},
\[ \gamma_1 \in (p_0(\beta_0-\alpha +1)-1, p_0(\beta_0-\alpha+2)-1) \quad \text{and} \quad \bar{\gamma}_1 =\gamma_1 + p_0 (\alpha /2-\beta_0) \]
with some $p_0 \in (1, \infty)$ and $\beta_0 \in (\alpha-1, \min\{1, \alpha\}]$. The main result of the subsection is Lemma \ref{oscil-lemma-2} below.

\smallskip
Let us point out that both $\mu_1$ and $\bar{\mu}_1$ depend on the choice of $\beta_0$, and
\begin{equation} \label{mu1=bar-mu-1}
\mu_1= \bar{\mu}_1 \quad \text{when} \quad \beta_0 = \alpha/2.
\end{equation}
To get the weighted estimate of $U$ in $L_p(\Omega_T, x_d^\gamma\, dz)$ with the optimal range for $\gamma$ as in Theorem \ref{thm:xd}, we will use $\beta_0 = \min\{1, \alpha\}$. On the other hand, to derive the estimate for $Du$, we will use $\beta_0 = \alpha/2$ and \eqref{mu1=bar-mu-1}.

\smallskip
For the reader's convenience, let us also recall that for a cylinder $Q \subset \bR^{d+1}$, a locally finite measure $\omega$, and an $\omega$-integrable function $g$ on $Q$, we denote the average of $g$ on $Q$ with respect to the measure $\omega$ by
\[
(g)_{Q, \omega} = \frac{1}{\omega(Q)}\int_{Q} g(z)\, \omega(dz)
\]
and  the average of $g$ on $Q$ with respect to the Lebesgue measure by
\[
 (g)_{Q} =\frac{1}{|Q|} \int_{Q} g(z)\, dz.
\]
We begin with the following lemma on the mean oscillation estimates of solutions to the homogeneous equations.
\begin{lemma} \label{oscil-lemma-1} Let $\nu \in (0,1)$, $\alpha \in (0,2)$, $p_0 \in (1, \infty)$, $\beta_0 \in {(\alpha-1}, \min\{1, \alpha\}]$, and $\gamma_1 \in (p_0(\beta_0-\alpha +1) -1, p_0(\beta_0-\alpha+2)-1)$. There exists $N = N(d,\nu, \alpha, \gamma_1, \beta_0, p_0)>0$ such that if $u \in \sW^{1,2}_{p_0}(Q_{14\rho}^+(z_0), x_d^{\gamma_1'}\, dz)$ is a strong solution of
\[
\left\{
\begin{array}{cccl}
\sL_0 u & =& 0 & \quad \text{in} \quad Q_{14\rho}^+(z_0) \\
u & = & 0 & \quad \text{on} \quad Q_{14\rho}(z_0) \cap \{x_d =0\}
\end{array}
\right.
\]
for some $\lambda>0, \rho>0$, $z_0 =(z_0', x_{d0}) \in \overline{\Omega}_T$, and for $\gamma_1' = \gamma_1 -p_0(\beta_0-\alpha)$, then
\begin{equation} \label{osc-h}
(|U - (U)_{Q_{\kappa \rho}^+(z_0), \mu_1}|)_{Q_{\kappa \rho}^+(z_0), \mu_1} \leq N \kappa^{\theta} (|U|^{p_0})_{Q_{14 \rho}^+(z_0), \mu_1}^{1/p_0}
\end{equation}
and
\begin{equation} \label{Du-osc-h}
(|Du - (Du)_{Q_{\kappa \rho}^+(z_0),  \bar{\mu}_1}|)_{Q_{\kappa \rho}^+(z_0), \bar{\mu}_1} \leq N \kappa^{\theta} (|Du|^{p_0})_{Q_{14 \rho}^+(z_0),  \bar{\mu}_1}^{1/p_0}
\end{equation}
for every $\kappa \in (0,1)$, where $\mu_1, \bar{\mu}_1$ are defined in \eqref{mu-1-def},
\[
U = (\chi^{-\beta_0} u_t,  \chi^{\alpha-\beta_0}DD_{x'} u, \lambda \chi^{-\beta_0}u),
\]
and $\theta = \min\{\beta_1(\alpha, \beta_0),  (2-\alpha)/4, 2-\alpha, 1\} \in (0,1)$ in which $\beta_1$ is defined in lemma \ref{prop:boundary}.
\end{lemma}
\begin{proof}
	By using the scaling \eqref{scaling}, we assume that $\rho =1$. We consider two cases: the boundary case and the interior one.

\noindent
\smallskip
{\em Boundary case.}  Consider $x_{0d} <4$. Let $\bar{z} =(t_0, x_0', 0)$ and note that from the definition of cylinders in \eqref{def:Q}, we have
\[
Q_{1}^+(z_0) \subset Q_{5}^+(\bar{z}_0) \subset  Q_{10}^+(\bar{z}_0) \subset Q_{14}^+(z_0).
\]
Then, we apply the mean value theorem  and the estimates \eqref{eq:hom-b3}-\eqref{eq:hom-b5} in Lemma \ref{prop:boundary} with  $\gamma_1$ in place of $\alpha_0$.  
We infer that
\[
\begin{split}
& (|U - (U)_{Q_{\kappa}^+(z_0), \mu_1}|)_{Q_{\kappa}^+(z_0), \mu_1} \\
& \leq N \kappa^{2-\alpha}
\|\partial_t U\|_{L^\infty(Q_1(z_0))} + N \kappa^{\beta_1} \llbracket U \rrbracket_{C^{0, \beta_1}(Q_{1}^+(z_0))} \\
& \leq N \kappa^{\theta}\big[ \|\partial_t U\|_{L^\infty(Q_{5}^+(\bar{z}))}  + \llbracket U \rrbracket_{C^{0, \beta_1}(Q_{5}^+(\bar{z}))}   \big] \\
&  \leq N \kappa^{\theta} (|U|^{p_0})_{Q_{10}^+(\bar{z}), \mu_1}^{1/p_0} \leq N \kappa^{\theta} (|U|^{p_0})_{Q_{14}^+(z_0), \mu_1}^{1/p_0},
\end{split}
\]
where we used the doubling property of $\mu_1$ in the last step.
This implies the estimate \eqref{osc-h} as $\kappa \in (0,1)$.  To estimate the oscillation of $Du$ as asserted in \eqref{Du-osc-h}, we note that
$$\bar{\gamma}_1 = \gamma_1 - p_0(\beta_0 -\alpha/2) > p_0 (1-\alpha/2)-1 >-1.  $$
Therefore,  \eqref{Du-osc-h} can be proved in a similar way as that of \eqref{osc-h} using the estimate \eqref{eq:hom-b-Du} in Lemma \ref{prop:boundary} with $\beta_0=0$ and $\alpha_0 = \bar{\gamma}_1 >-1$.

\ \\ \noindent
\smallskip
{\em Interior case.} Consider $x_{0d} >4\rho=4$. By using Lemma  \ref{prop:int} with $\beta = -\beta_0$ and the doubling property of $\mu_1$, we see that
\[
\begin{split}
& (|\chi^{-\beta_0} u - (\chi^{-\beta_0} u)_{Q_{\kappa}^+(z_0), \mu_1}|)_{Q_{\kappa}^+(z_0), \mu_1} \\
& \leq N \kappa^{1/2-\alpha/4}\llbracket \chi^{-\beta_0} u \rrbracket_{C_\alpha^{1/4, 1/2}(Q_{1}^+(z_0))} \\
& \leq N \kappa^{1/2-\alpha/4}   \left(\fint_{Q_{2}^+(z_0)}|\chi^{-\beta_0} u|^{p_0}\mu_1(dz) \right)^{1/p_0} \\
& \leq N \kappa^{1/2-\alpha/4}   \left(\fint_{Q_{14}^+(z_0)}|\chi^{-\beta_0} u|^{p_0} \mu_1(dz) \right)^{1/p_0}.
\end{split}
\]
Similarly, by using the finite difference quotient, we can apply Lemma \ref{prop:int} to $u_t$ and obtain
\[
\begin{split}
& (|\chi^{-\beta_0} u_t - (\chi^{-\beta_0} u_t)_{Q_{\kappa}^+(z_0), \mu_1}|)_{Q_{\kappa}^+(z_0), \mu_1} \\
& \leq N \kappa^{1/2-\alpha/4}   \left(\fint_{Q_{14}^+(z_0)}|\chi^{-\beta_0} u_t|^{p_0} \mu_1(dz) \right)^{1/p_0}.
\end{split}
\]
In the same way, by applying Lemma \ref{prop:int} to $D_{x'}u$ with $\beta=\alpha-\beta_0$ and $\alpha_0 = \gamma_1$, we infer that
\[
\begin{split}
& (|\chi^{\alpha -\beta_0} DD_{x'} u  - (\chi^{\alpha -\beta_0} DD_{x'} u)_{Q_{\kappa}^+(z_0), \mu_1}|)_{Q_{\kappa}^+(z_0), \mu_1} \\
& \leq N \kappa^{1/2-\alpha/4}   \left(\fint_{Q_{14}^+(z_0)}|\chi^{\alpha -\beta_0} DD_{x'} u|^{p_0} \mu_1(dz) \right)^{1/p_0}.
\end{split}
\]
The oscillation estimate of $Du$ can be proved in a similar way. Therefore, we obtain \eqref{osc-h}. The proof of the lemma is completed.
\end{proof}

Now, we recall that for a given number $a \in \bR$,
$a_+ = \max\{a, 0\}$.
We derive the oscillation estimates of solutions to the non-homogeneous equation \eqref{eq:xd}, which is the main result of the subsection.
\begin{lemma} \label{oscil-lemma-2} Let $\nu \in (0,1)$, $\alpha \in (0,2)$, $p_0 \in (1, \infty)$, $\beta_0 \in {(\alpha-1},  \min\{1, \alpha\}]$, and $\gamma_1 \in (p_0(\beta_0-\alpha +1) -1, p_0(\beta_0-\alpha+2)-1)$. There exists $N = N(d,\nu, \alpha, \gamma_1, \beta_0, p_0)>0$ such that the following assertions hold. Suppose that  $u \in  \sW^{1,2}_{p_0, \textup{loc}}(\Omega_T, x_d^{\gamma_1'}\, dz)$ is a strong solution of \eqref{eq:xd} with $f \in L_{p_0, \textup{loc}}(\Omega_T, x_d^{\gamma_1'}\, dz)$ and $\gamma_1' = \gamma_1 - p_0(\beta_0-\alpha)$. Then, for every $z_0 \in \overline{\Omega}_T$, $\rho \in (0, \infty)$,  $\kappa \in (0,1)$, we have
\[
\begin{split}
& (|U - (U)_{Q_{\kappa \rho}^+(z_0), \mu_1}|)_{Q_{\kappa \rho}^+(z_0), \mu_1}\\
&\leq N  \kappa^{\theta} (|U|^{p_0})_{Q_{14 \rho}^+(z_0), \mu_1}^{1/p_0} + N \kappa^{-(d + (\gamma_1)_+ +2-\alpha)/p_0}  (|\chi^{\alpha-\beta_0} f|^{p_0})_{Q_{14\rho}^+(z_0), \mu_1}^{1/p_0}
\end{split}
\]
and
\[
\begin{split}
& \lambda^{1/2}(|Du - (Du)_{Q_{\kappa \rho}^+(z_0), \bar{\mu}_1}|)_{Q_{\kappa \rho}^+(z_0),\bar{\mu}_1} \\
& \leq  N  \kappa^{\theta} \lambda^{1/2} (|Du|^{p_0})_{Q_{14 \rho}^+(z_0), \bar{\mu}_1}^{1/p_0} + N \kappa^{-(d + (\gamma_1)_+ +2-\alpha)/p_0}  (|\chi^{\alpha/2} f|^{p_0})_{Q_{14\rho}^+(z_0),  \bar{\mu}_1}^{1/p_0},
\end{split}
\]
where $\theta{\in (0,1)}$ is defined in Lemma \ref{oscil-lemma-1},
\[
U = (\chi^{-\beta_0} u_t,  \chi^{\alpha-\beta_0}DD_{x'} u,  \lambda \chi^{-\beta_0}u),
\]
and $\mu_1, \bar{\mu}_1$ are defined in \eqref{mu-1-def}.
\end{lemma}
\begin{proof} As $\gamma_1 \in (p_0(\beta_0-\alpha +1)-1, p_0(\beta_0-\alpha+2)-1)$, we see that
\[
\gamma_1' = \gamma_1 - p_0(\beta_0-\alpha) \in (p_0 -1, 2p_0 -1).
\]
Therefore, by Lemma \ref{l-p-sol-lem}, there is a strong solution $v \in \sW^{1,2}_{p_0}(\Omega_T, x_d^{\gamma_1'}\, dz)$ to
\begin{equation} \label{v-sol-1}
\left\{
\begin{array}{cccl}
\sL_0 v & =& f \mathbf{1}_{Q_{14\rho}^+(z_0)} & \quad \text{in} \quad \Omega_T,\\
v & = & 0 & \quad \text{on} \quad \{x_d =0\}
\end{array} \right.
\end{equation}
satisfying
\begin{align} \notag
& \|\chi^{-\alpha} v_t\|_{L_{p_0}(\Omega_T, x_d^{\gamma_1'} dz)}  + \|D^2 v\|_{L_{p_0}(\Omega_T, x_d^{\gamma_1'} dz)}  + \lambda^{1/2} \|\chi^{-\alpha/2} Dv\|_{L_{p_0}(\Omega_T, x_d^{\gamma_1'} dz)} \\ \label{v-sol-est-1}
& \qquad + \lambda \|\chi^{-\alpha} v\|_{L_{p_0}(\Omega_T, x_d^{\gamma_1'} dz)}  \leq N \|f\|_{L_{p_0}(Q_{14\rho}^+(z_0), x_d^{\gamma_1'} dz)}.
\end{align}
Let us denote
\[
  V = (\chi^{-\beta_0} v_t,  \chi^{\alpha-\beta_0}DD_{x'} v, \lambda \chi^{-\beta_0}v).
\]
Then, it follows from \eqref{v-sol-est-1} and the definitions of $\mu_1$ and $\gamma_1'$ that
\begin{equation} \label{V-osc-est-1}
(|V|^{p_0})_{Q_{14\rho}^+(z_0), \mu_1}^{1/p_0} \leq N (|\chi^{\alpha-\beta_0} f|^{p_0})_{Q_{14\rho}^+(z_0), \mu_1}^{1/p_0}.
\end{equation}
Note also that due to \eqref{v-sol-est-1} and the definition of $\bar{\gamma}_1$,
\[
\begin{split}
\lambda^{1/2} \left( \int_{Q_{14\rho}^+(z_0)} | Dv|^{p_0}  x_d^{\bar{\gamma}_1}dz \right)^{1/p_0}  & = \lambda^{1/2} \left( \int_{Q_{14\rho}^+(z_0)} |\chi^{-\alpha/2}Dv|^{p_0} x_d^{\gamma_1'}  dz \right)^{1/p_0} \\
& \leq N \left( \int_{Q_{14\rho}^+(z_0)} |\chi^{ \alpha/2} f|^{p_0} x_d^{\bar{\gamma}_1}dz \right)^{1/p_0}.
 \end{split}
\]
Then,
\begin{equation} \label{os-Dv-18}
 \lambda^{1/2} (|Dv|^{p_0})_{Q_{14\rho}^+(z_0), \bar{\mu}_1}^{1/p_0} \leq N (|\chi^{\alpha/2}f|^{p_0})_{Q_{14\rho}^+(z_0), \bar{\mu}_1}^{1/p_0}.
\end{equation}
Now, let $w = u- v$. From \eqref{v-sol-1}, we see that $w \in \sW^{1,2}_{p_0}(Q_{14\rho}^+(z_0), x_d^{\gamma_1'}\, dz)$ is a strong solution of
\[
\left\{
\begin{array}{cccl}
\sL_0 w & =& 0 & \quad \text{in} \quad Q_{14\rho}^+(z_0),\\
w & = & 0 & \quad \text{on} \quad Q_{14\rho}(z_0) \cap \{x_d =0\}.
\end{array}
\right.
\]
Then, by applying Lemma \ref{oscil-lemma-1} to $w$, we see that
\begin{equation} \label{W-osc-est-1}
(|W - (W)_{Q_{\kappa \rho}^+(z_0), \mu_1}|)_{Q_{\kappa \rho}^+(z_0), \mu_1} \leq N \kappa^{\theta} (|W|^{p_0})_{Q_{14 \rho}^+(z_0), \mu_1}^{1/p_0}
\end{equation}
and
\begin{equation} \label{os-Dw-18}
(|Dw - (Dw)_{Q_{\kappa \rho}^+(z_0), \bar{\mu}_1}|)_{Q_{\kappa \rho}^+(z_0), \bar{\mu}_1} \leq N \kappa^{\theta} (|Dw|^{p_0})_{Q_{14 \rho}^+(z_0), \bar{\mu}_1}^{1/p_0},
\end{equation}
where
\[
\begin{split}
& W= (\chi^{-\beta_0} w_t,  \chi^{\alpha-\beta_0}DD_{x'} w,  \lambda \chi^{-\beta_0}w).
\end{split}
\]
Now, note that from \eqref{def:Q} and \eqref{def:r} we have
\begin{align} \notag
\frac{\mu_1(Q_{14 \rho}^+(z_0))}{\mu_1(Q_{\kappa \rho}^+(z_0))} & = N(d) \kappa^{\alpha-2} \Big(\frac{r(14\rho, x_{0d})}{r(\kappa \rho, x_{0d})}\Big)^{d + (\gamma_1)_+} \\  \label{Q-compared}
& \leq N (d)\kappa^{-(d+ (\gamma_1)_+ + 2-\alpha)}.
\end{align}
Then, it follows from the triangle inequality, H\"{o}lder's inequality, \eqref{W-osc-est-1}, and \eqref{Q-compared}  that
\[
\begin{split}
& (|U - (U)_{Q_{\kappa \rho}^+(z_0), \mu_1}|)_{Q_{\kappa \rho}^+(z_0), \mu_1} \\
& \leq (|W - (W)_{Q_{\kappa \rho}^+(z_0), \mu_1}|)_{Q_{\kappa \rho}^+(z_0), \mu_1} + (|V - (V)_{Q_{\kappa \rho}^+(z_0), \mu_1}|)_{Q_{\kappa \rho}^+(z_0), \mu_1} \\
& \leq (|W - (W)_{Q_{\kappa \rho}^+(z_0), \mu_1}|)_{Q_{\kappa \rho}^+(z_0), \mu_1} \\
& \qquad \quad+ N(d) \kappa^{-(d+ (\gamma_1)_+ +2-\alpha)/p_0} (|V|^{p_0})^{1/p_0}_{Q_{14\rho}^+(z_0), \mu_1}\\
& \leq N  \kappa^{\theta} (|W|^{p_0})_{Q_{14 \rho}^+(z_0), \mu_1}^{1/p_0} + N(d) \kappa^{-(d + (\gamma_1)_+ +2-\alpha)/p_0} (|V|^{p_0})^{1/p_0}_{Q_{14\rho}^+(z_0), \mu_1}.
\end{split}
\]
As $W = U -V$ and $\kappa \in (0,1)$, we apply the triangle inequality again to the term involving $W$ on the right-hand side of the last estimate to see that
\[
\begin{split}
& (|U - (U)_{Q_{\kappa \rho}^+(z_0), \mu_1}|)_{Q_{\kappa \rho}^+(z_0), \mu_1}\\
&\leq N  \kappa^{\theta} (|U|^{p_0})_{Q_{14 \rho}^+(z_0) , \mu_1}^{1/p_0}  + N\big( \kappa^{-(d + (\gamma_1)_+ +2-\alpha)/p_0} + \kappa^{\theta}\big)(|V|^{p_0})_{Q_{14\rho}^+(z_0) , \mu_1}^{1/p_0} \\
& \leq N  \kappa^{\theta} (|U|^{p_0})_{Q_{14 \rho}^+(z_0), \mu_1}^{1/p_0} + N \kappa^{-(d + (\gamma_1)_+ +2-\alpha)/p_0}  (|V|^{p_0})_{Q_{14\rho}^+(z_0), \mu_1}^{1/p_0}.
\end{split}
\]
From this and \eqref{V-osc-est-1}, it follows that
\[
\begin{split}
& (|U - (U)_{Q_{\kappa \rho}^+(z_0) , \mu_1}|)_{Q_{\kappa \rho}^+(z_0), \mu_1}\\
& \leq N  \kappa^{\theta} (|U|^{p_0})_{Q_{14 \rho}^+(z_0) , \mu_1}^{1/p_0} + N \kappa^{-(d + (\gamma_1)_++2-\alpha)/p_0}  (|\chi^{\alpha-\beta_0} f|^{p_0})_{Q_{14\rho}^+(z_0) , \mu_1}^{1/p_0},
\end{split}
\]
where $N = N(d,\nu, \alpha, \gamma_1, \beta_0, p_0)>0$. This proves the assertion on the oscillation of $U$. The oscillation estimate of $Du$ can be proved similarly using \eqref{os-Dv-18} and \eqref{os-Dw-18}. The proof of the lemma is completed.
\end{proof}

We now conclude this subsection by pointing out the following important remark, whose proof can be achieved by following that of Lemma \ref{oscil-lemma-2} with minor modifications.
\begin{remark} \label{all-oscilla-est} Under the assumptions as in Lemma \ref{oscil-lemma-2}, and  if $\beta_0 \in {(\alpha-1}, \alpha/2]$, it holds that
\[
\begin{split}
& (|U' - (U')_{Q_{\kappa \rho}^+(z_0), \mu_1}|)_{Q_{\kappa \rho}^+(z_0), \mu_1}\\
&\leq N  \kappa^{\theta} (|U'|^{p_0})_{Q_{14 \rho}^+(z_0), \mu_1}^{1/p_0} + N \kappa^{-(d + (\gamma_1)_+ +2-\alpha)/p_0}  (|\chi^{\alpha-\beta_0} f|^{p_0})_{Q_{14\rho}^+(z_0), \mu_1}^{1/p_0}
\end{split}
\]
where
\[
U' = (\chi^{-\beta_0} u_t,  \chi^{\alpha-\beta_0}DD_{x'} u,  \lambda^{1/2}\chi^{\alpha/2-\beta_0} Du, \lambda \chi^{-\beta_0}u).
\]
\end{remark}
\subsection{Proof of Theorem \ref{thm:xd}} \label{proof-xd} We are now ready to give the proof of Theorem \ref{thm:xd}.
\begin{proof}[Proof of Theorem \ref{thm:xd}] We begin with the proof of the a priori estimates \eqref{eq:xd-main}--\eqref{eq3.09} assuming that $u \in \sW^{1,2}_p(\Omega_T, x_d^\gamma\, dz)$ is a strong solution to the equation \eqref{eq:xd} with
\begin{equation} \label{gamma-alla-range}
\gamma \in (p (\alpha-1)_{+} -1, 2p-1), \quad \text{where} \,\, (\alpha -1)_+ = \max\{\alpha-1, 0\}.
\end{equation}
In our initial step, we prove \eqref{eq:xd-main}--\eqref{eq3.09} with an extra assumption that $u$ is compactly supported.  We first prove \eqref{eq:xd-main}. Let $\beta_0 = \min\{1, \alpha\}$, and we will apply Lemma \ref{oscil-lemma-2} with this $\beta_0$. Let $p_0 \in (1, p)$ and $\gamma_1 \in (p_0(\beta_0-\alpha +1) -1, p_0(\beta_0-\alpha+2)-1)$. We choose $p_0$ to be sufficiently close to $1$ and $\gamma_1$ to be sufficiently close to $p_0(\beta_0-\alpha+2)-1$ so that
\begin{equation} \label{nature-choice-2}
\gamma - [\gamma_1 +p(\alpha-\beta_0)]  < (1+\gamma_1)(p/p_0 -1).
\end{equation}
We note that this is possible  because $\alpha-\beta_0 = (\alpha-1)_+$ and
\[
\gamma - [ \gamma_1 + p(\alpha-\beta_0)] < p[2  - (\alpha-1)_+] -1 -\gamma_1,
\]
and also from our choices of $p_0$ and $\gamma_1$,
\[
\begin{split}
(1+\gamma_1)(p/p_0 -1) & \sim p (1+\gamma_1) -1 -\gamma_1 \sim p[2  - (\alpha-1)_+] -1 -\gamma_1.
\end{split}
\]
Now, let us denote
\begin{equation} \label{gamma-1-pri}
\gamma_1' : = \gamma_1 + p(\alpha-\beta_0) = \gamma_1 + p (\alpha-1)_+.
\end{equation}
Due to \eqref{gamma-alla-range} and the definition of $\gamma_1'$,  it  follows that
\begin{equation} \label{nature-choice-1}
\gamma - \gamma_1'  = \gamma - p(\alpha-1)_+ - \gamma_1 > -1 - \gamma_1.
\end{equation}
From \eqref{nature-choice-1} and \eqref{nature-choice-2}, it holds that
\begin{equation} \label{gamma-0414}
\gamma' : =  \gamma -  \gamma_1' \in (-1-\gamma_1, (1+\gamma_1)(p/p_0 -1)).
\end{equation}

\smallskip
Now, since $u$ has compactly support in $\Omega_T$, we have $u \in \sW^{1,2}_p(\Omega_T, x_d^{\gamma_1'} dz)$. Therefore, it follows from Lemma \ref{oscil-lemma-2} that
\[
U^{\#}_{\mu_1} \leq N \Big[ \kappa^{\theta} \cM_{\mu_1}(|U|^{p_0}) ^{1/p_0} +   \kappa^{-(d+ (\gamma_1)_+ + 2-\alpha)/2}  \cM_{\mu_1}(|\chi^{\alpha-\beta_0} f|^{p_0})^{1/p_0} \Big],
\]
where $\mu_1(dz) = x_d^{\gamma_1}dxdt$, the sharp function and the maximal function with respect to the measure $\mu_1$ are defined as in Subsection \ref{sharp-function-sec}, and
\[
U = (\chi^{-\beta_0} u_t,  \chi^{\alpha-\beta_0}DD_{x'} u,  \lambda \chi^{-\beta_0}u).
\]
Next, due to \eqref{gamma-0414}, $x_d^{\gamma'} \in A_{p/p_0}(\mu_1)$.  It then follows from the weighted Fefferman-Stein theorem and Hardy-Littlewood theorem (i.e., Theorem \ref{FS-thm}) that
\begin{align} \notag
&  \|U\|_{L_p(\Omega_T, x_d^{\gamma'}\, d\mu_1)} \leq N \|U^{\#}_{\mu_1}\|_{L_p(\Omega_T, x_d^{\gamma'}\,d\mu_1)} \\ \notag
& \leq N \Big[\kappa^{\theta}\|\cM_{\mu_1}(|U|^{p_0})^{1/p_0}\|_{L_p(\Omega_T, x_d^{\gamma'}\,d\mu_1 )}  \\ \notag
& \qquad +   \kappa^{-(d + (\gamma_1)_++2-\alpha)/2}  \|\cM_{\mu_1}(|\chi^{\alpha-\beta_0} f|^{p_0})^{1/p_0}\|_{L_p(\Omega_T, x_d^{\gamma'}\, d\mu_1)} \Big] \\ \label{est:0414-1}
& \leq N \Big[ \kappa^{\theta} \|U\|_{{L_p(\Omega_T, x_d^{\gamma'}\,d\mu_1)}} + \kappa^{-(d+ (\gamma_1)_++2-\alpha)/2}  \| \chi^{\alpha-\beta_0} f\|_{{L_p(\Omega_T, x_d^{\gamma'}\,d\mu_1)}} \Big].
\end{align}
From the definition of $U$, the choices of $\gamma'$ in \eqref{gamma-0414} and $\gamma_1'$ in \eqref{gamma-1-pri}, we have
\[
\begin{split}
\|U\|_{{L_p(\Omega_T, x_d^{\gamma'}d\mu_1)}} & = \|\chi^{-\alpha} u_t\|_{L_p(\Omega_T, x_d^\gamma\, dz)} + \|DD_{x'}u\|_{L_p(\Omega_T, x_d^\gamma dz)} \\
& \quad + \lambda  \|\chi^{-\alpha} u \|_{L_p(\Omega_T, x_d^\gamma\, dz)}
<\infty.
\end{split}
\]
Then, by choosing $\kappa \in (0,1)$ sufficiently small so that $N \kappa^{\theta} < 1/2$, we obtain from \eqref{est:0414-1} that
\[
\begin{split}
& \|\chi^{-\alpha} u_t\|_{L_p(\Omega_T, x_d^\gamma\, dz)} + \|DD_{x'}u\|_{L_p(\Omega_T, x_d^\gamma\, dz)} + \lambda  \|\chi^{-\alpha} u \|_{L_p(\Omega_T, x_d^\gamma\, dz)}  \\
&  \leq N \| f\|_{{L_p(\Omega_T, x_d^{\gamma'}\,d\mu_1)}} = N \| f\|_{{L_p(\Omega_T, x_d^{\gamma}\,dz)}}.
\end{split}
\]
Also, from the PDE in \eqref{eq:xd}, we see that
\[
|D_{dd} u| \leq N[|DD_{x'}u| + (|u_t| + \lambda |u|)x_d^{-\alpha} + |f|],
\]
and therefore
\[
\begin{split}
& \|\chi^{-\alpha}u_t\|_{L_p(\Omega_T, x_d^{\gamma}\,dz)}+\lambda\|\chi^{-\alpha}u\|_{L_p(\Omega_T, x_d^{\gamma}\,dz)}
+\|D^2 u\|_{L_p(\Omega_T, x_d^{\gamma}\,dz)}\\
&\le N\|f\|_{L_p(\Omega_T,  x_d^{\gamma}\,dz)},
\end{split}
\]
which is \eqref{eq:xd-main}.

\smallskip
Next, we prove the estimate \eqref{eq3.09} also with the extra assumption that $u$ has compact support. We observe that if $\gamma \in (p -1, 2p -1)$,  \eqref{eq3.09} follows from \eqref{est-0405-1}. Therefore, it remains to consider the case that  $\gamma \in (\alpha p/2-1, p  -1]$ or equivalently
\begin{equation} \label{gamma-range-2}
\gamma - \alpha p/2 \in (-1, p(1-\alpha/2) -1].
\end{equation}
The main idea is to apply  Lemma \ref{oscil-lemma-2} with this $\beta_0 = \alpha/2$. Let $p_0, \gamma_1$ be as before but with the new choice of $\beta_0$. As noted in \eqref{mu1=bar-mu-1}, we have
\[ \bar{\gamma}_1 = \gamma_1 - p_0(\beta_0 -\alpha/2) = \gamma_1 \qquad \text{and} \qquad  \bar{\mu}_1 = \mu_1.
\]
Because of \eqref{gamma-range-2},  we can perform the same calculation  as the one that yields \eqref{gamma-0414} to obtain
\[
\bar{\gamma}' := \gamma - (\bar{\gamma}_1 + p\alpha/2 ) \in (-1 - \bar{\gamma}_1, (1+\bar{\gamma}_1)(p/p_0 -1))
\]
and therefore $x_d^{\bar{\gamma}'} \in A_{p/p_0}(\bar{\mu}_1)$.  By using Lemma \ref{oscil-lemma-2} , we have
\begin{equation} \label{Du-sharp}
\begin{split}
 \lambda^{1/2} (Du)^{\#}_{\bar{\mu}_1} & \leq N \Big[ \kappa^{\theta}  \lambda^{1/2}\cM_{\bar{\mu}_1}(|Du|^{p_0}) ^{1/p_0} \\
& \quad +   \kappa^{-(d+ \bar{\gamma}_1 + 2-\alpha)/2}  \cM_{\bar{\mu}_1}(|\chi^{\alpha/2} f|^{p_0})^{1/p_0} \Big],
\end{split}
\end{equation}
where $\bar{\mu}_1(dz) = x_d^{\bar{\gamma}_1}dxdt$.  We apply Theorem \ref{FS-thm} to \eqref{Du-sharp}, and  then choose $\kappa>0$ sufficiently small as in the proof of \eqref{eq:xd-main} to obtain
\[
\lambda^{1/2}\|Du\|_{L_p(\Omega_T, x_d^{\bar{\gamma}'} d\bar{\mu}_1)} \leq N \|\chi^{\alpha/2} f\|_{L_p(\Omega_T, x_d^{\bar{\gamma}'}\, d\bar{\mu}_1)}.
\]
This implies
\[
\lambda^{1/2} \|\chi^{-\alpha/2}Du\|_{L_p(\Omega_T, x_d^{\gamma}\, dz)} \leq N \| f\|_{L_p(\Omega_T, x_d^{\gamma}\,  dz)}
\]
as $\gamma -  p\alpha/2= \bar{\gamma}' + \bar{\gamma}_1$. The estimate  \eqref{eq3.09} is proved.

\smallskip
Now, we prove  \eqref{eq:xd-main}--\eqref{eq3.09} without the assumption that $u$ is compactly supported. As $u \in \sW^{1,2}_p(\Omega_T, x_d^\gamma dz)$, there is a sequence $\{u_n\}$ in $C_0^\infty(\Omega_T)$ such that
\begin{equation} \label{u-approx-compact}
\lim_{n\rightarrow \infty} \|u_n -u\|_{\sW^{1,2}_p(\Omega_T, x_d^\gamma\,  dz)} =0.
\end{equation}
Let $f_n =  f + \sL_0 (u_n - u)/\mu(x_d)$ and observe that $u_n$ is a strong solution of
\[
\sL_0 u_n = \mu(x_d) f_n \quad \text{in} \quad \Omega_T \quad \text{and} \quad u_n =0 \quad \text{on} \quad \{x_d =0\}.
\]
Then, applying the estimates \eqref{eq:xd-main}--\eqref{eq3.09} to $u_n$, we obtain
\begin{equation} \label{un-supported}
\|u_n\|_{\sW^{1,2}_p(\Omega_T, x_d^\gamma\, dz)} \leq N\|f_n\|_{L_p(\Omega_T, x_d^\gamma\, dz)}.
\end{equation}
Note that
\[
\begin{split}
& \|f_n\|_{L_p(\Omega_T, x_d^\gamma\, dz)}  \leq \|f\|_{L_p(\Omega_T, x_d^\gamma\, dz)} + N\lambda \|\chi^{-\alpha} (u-u_n)\|_{L_p(\Omega_T, x_d^\gamma\, dz)}\\
& \qquad + N \Big [ \|D^2(u-u_n)\|_{L_p(\Omega_T, x_d^\gamma\, dz)}  + \|\chi^{-\alpha}(u-u_n)_t\|_{L_p(\Omega_T, x_d^\gamma\, dz)}  \| \Big] \\
& \rightarrow \|f\|_{L_p(\Omega_T, x_d^\gamma\, dz)} \quad \text{as} \quad n \rightarrow \infty.
\end{split}
\]
Therefore, by taking $n\rightarrow \infty$ in \eqref{un-supported} and using \eqref{u-approx-compact},  we obtain the estimates \eqref{eq:xd-main}--\eqref{eq3.09} for $u$. Hence, the proof of \eqref{eq:xd-main}--\eqref{eq3.09} is completed.

\smallskip
It remains to prove the existence of a strong solution $u \in \sW^{1,2}_p(\Omega_T, x_d^{\gamma}\, dz)$ to  \eqref{eq:xd} assuming that $f \in L_p(\Omega_T, x_d^\gamma\, dz)$, for $p \in (1, \infty)$ and $\gamma \in (p (\alpha-1)_{+} -1, 2p-1)$. We observe when $\gamma \in (p-1, 2p-1)$, the existence of solution is already proved in Lemma \ref{l-p-sol-lem}.  Therefore, it remains to consider the case when
$$\gamma \in (p (\alpha-1)_{+} -1, p-1].$$
We consider two cases.

\smallskip
\noindent
{\em Case} 1.
Consider $\gamma \in (p(\alpha-1)_+ -1, p-1)$.
As $f \in L_p(\Omega_T, x_d^\gamma\, dz)$, there is a sequence $\{f_k\}_k \subset C_0^\infty(\Omega_T)$ such that
\begin{equation} \label{fk-approx}
\lim_{k\rightarrow \infty}\|f_k - f\|_{L_p(\Omega_T, x_d^\gamma\, dz)} =0.
\end{equation}
For each $k \in \bN$, because $f_k$ has compact support, we see that
\[  x_d^{1-\alpha} \mu(x_d) f_k \sim x_d f_k \in L_p(\Omega_T, x_d^{\gamma}\, dz). \]
Then, as in the proof of Lemma \ref{l-p-sol-lem}, we apply \cite[Theorem 2.4]{DPT21} to find a weak solution $u_k \in \cH^1_p(\Omega_T, x_d^{\gamma}\, dz)$ to the divergence form equation \eqref{eq:xd-div} with $f_k$ in place of $f$. Moreover,
\begin{equation} \label{uk-Hp-est}
\|Du_k\|_{L_p(\Omega_T, x_d^\gamma\, dz)} +  \|\chi^{-\alpha/2}u_k\|_{L_p(\Omega_T, x_d^\gamma\, dz)} < \infty.
\end{equation}
We claim that $u_k \in \sW^{1,2}_p(\Omega_T, x_d^{\gamma}\, dz)$ for each $k \in \mathbb{N}$.  Note that if the claim holds, we can apply the a priori estimate that we just proved for the equations of $u_k$ and of $u_k - u_l$ to get
\[
\begin{split}
& \|u_k\|_{\sW^{1,2}_p(\Omega_T, x_d^{\gamma}\, dz)} \leq N \|f_k\|_{L_p(\Omega_T, x_d^{\gamma}\, dz)} \quad \text{and} \\
& \|u_k - u_l\|_{\sW^{1,2}_p(\Omega_T, x_d^{\gamma}\, dz)} \leq N \|f_k - f_l\|_{L_p(\Omega_T, x_d^{\gamma}\, dz)}
\end{split}
\]
for any $k, l \in \mathbb{N}$, where $N = N(d, \nu, \gamma, \alpha, p)>0$ which is independent of $k, l$. The last estimate and \eqref{fk-approx} then imply that the sequence $\{u_k\}_k$ is convergent in $\sW^{1,2}_p(\Omega_T, x_d^{\gamma}\, dz)$. Let $u \in \sW^{1,2}_p(\Omega_T, x_d^{\gamma}\, dz)$ be the limit of such sequence, we see that $u$ solves \eqref{eq:xd}.

\smallskip
Hence, in this case, it remains to prove the claim that $u_k \in \sW^{1,2}_p(\Omega_T, x_d^{\gamma}\, dz)$ for every $k \in \mathbb{N}$.  Also, let us fix $k \in \mathbb{N}$, and let us denote $\Omega_T' = (-\infty, T) \times \bR^{d-1}$.  Let $0 < r_0 <R_0$ such that
\begin{equation} \label{fk-support}
\text{supp}(f_k) \subset {\Omega_T'} \times (r_0, R_0).
\end{equation}
Without loss of generality, we assume that $
r_0 =2$.
From \eqref{uk-Hp-est}, it follows directly that
\[
\begin{split}
& \|Du_k\|_{L_p(\Omega_T' \times (1, \infty), x_d^{\gamma -p}\, dz)} +   \|u_k\|_{L_p(\Omega_T' \times (1, \infty), x_d^{\gamma-2p}\, dz)} \\
& \qquad +   \|\chi^{-\alpha/2}u_k\|_{L_p(\Omega_T' \times (1,\infty), x_d^{\gamma-p}\, dz)} <\infty.
\end{split}
\]
Then, we can follow the proof of Lemma \ref{l-p-sol-lem} to show that
\[
\|u_k\|_{\sW^{1,2}_{p}(\Omega_T' \times (1,\infty), x_d^{\gamma}\, dz)} <\infty.
\]
It now remains to prove that $u_k \in \sW^{1,2}_{p}({\Omega_T'\times (0,1)}, x_d^{\gamma}\, dz)$ and
\begin{equation} \label{near-est-uk}
\|u_k\|_{\sW^{1,2}_{p}(\Omega_T' \times (0, 1), x_d^{\gamma}\, dz)} < \infty.
\end{equation}
To this end, because of \eqref{fk-support}, we note that $u_k$   solves the homogeneous equation
\begin{equation} \label{uk-ne-zero}
\sL_0 u_k =0 \quad \text{in} \quad  \Omega_T' \times (0, 2)
\end{equation}
with the boundary condition $u_k =0$ on $\{x_d =0\}$.  Let us denote
\[
\begin{split}
& C_r = [-1, 0) \times \big\{ x = (x_1, \ldots, x_d) \times \bR^{d}_+ : {\max_{1 \leq i \leq d}|x_i|}<r\big\}, \\
& C_r(t,x) = C_r + (t,x), \quad r >0.
\end{split}
\]
Consider $\alpha \in (0, 1)$. By using Lemmas \ref{caccio}, and \ref{prop:boundary} with a scaling argument and translation, we obtain
\begin{equation*}
\begin{split}
& \|\chi^{-\alpha} u_k\|_{L_\infty(C_{1}(z_0))} + \|Du_k\|_{L_\infty(C_{1}(z_0))}  + \|\chi^{-\alpha} \partial_t u_k\|_{L_\infty(C_{1}(z_0))}  \\
& \quad + \|DD_{x'} u_k\|_{L_\infty(C_{1}(z_0))} \leq N \Big[\|Du_k\|_{L_{p}(C_{2}(z_0), x_d^\gamma\, dz)} + \|\chi^{-\alpha/2} u_k\|_{L_{p}(C_{2}(z_0), x_d^\gamma dz)} \Big]
\end{split}
\end{equation*}
for every $z_0 = (t_0, x_0', 0) \in \Omega_T' \times\{0\}$. Note that $N$ depends on $k$, but is independent of $z_0$. This and the PDE in \eqref{uk-ne-zero} imply that
\[
\begin{split}
& \|\chi^{-\alpha} u_k\|_{L_\infty(C_{1}(z_0))} + \|Du_k\|_{L_\infty(C_{1}(z_0))} + \|\chi^{-\alpha} \partial_t u_k\|_{L_\infty(C_{1}(z_0))}   \\
& \quad + \|D^2 u_k\|_{L_\infty(C_{1}(z_0))}  \leq  N \Big[\|Du_k\|_{L_{p}(C_{2}(z_0), x_d^\gamma\, dz)} + \|\chi^{-\alpha/2} u_k\|_{L_{p}(C_{2}(z_0), x_d^\gamma\, dz)} \Big].
\end{split}
\]
Then,  as $\gamma > -1$, we see that
\[
\begin{split}
& \|\chi^{-\alpha}u_k\|_{L_p(C_1(z_0), x_d^{\gamma}\, dz)}
+ \|Du_k\|_{L_p(C_1(z_0), x_d^{\gamma}\, dz)}
+ \|\chi^{-\alpha}\partial_t u_k\|_{L_p(C_1(z_0), x_d^{\gamma}\, dz)}  \\
& + \|\ D^2u_k\|_{L_p(C_1(z_0), x_d^{\gamma}\, dz)} \leq N \Big[\|Du_k\|_{L_{p}(C_{2}(z_0), x_d^\gamma\, dz)} + \|\chi^{-\alpha/2} u_k\|_{L_{p}(C_{2}(z_0), x_d^\gamma\, dz)} \Big].
\end{split}
\]
Then, with $z_0 = (t_0, x_0', 0)$ and with $\mathcal{I} = ((\mathbb{Z}+T) \cap (-\infty, T{]}) \times (2\mathbb{Z})^{d-1}$, we have
\[
\begin{split}
\|u_k\|^p_{\sW^{1,2}_p(\Omega_T' \times (0,1))} & = \sum_{(t_0, x_0') \in \mathcal{I} } \|u_k\|^p_{\sW^{1,2}_p(C_1(z_0))} \\
& \leq N \sum_{(t_0, x_0') \in \mathcal{I} }\Big[\|Du_k\|^p_{L_{p}(C_{2}(z_0))} + \|\chi^{-\alpha/2} u_k\|^p_{L_{p}(C_{2}(z_0))} \Big] \\
& = N\Big[ \|Du_k\|^p_{L_{p}(\Omega_T, x_d^\gamma\, dz)} + \|\chi^{-\alpha/2} u_k\|^p_{L_{p}(\Omega_T, x_d^\gamma\, dz)}\Big] <\infty.
\end{split}
\]
Hence, \eqref{near-est-uk} holds.

\smallskip
Now, we consider $\alpha \in [1, 2)$.
As $\gamma + p (1-\alpha) >-1$, we see that
\[
\begin{split}
& \int_{C_1(z_0)} |x_d^{-\alpha} u_k(z)|^p x_d^\gamma dz  = \int_{C_1(z_0)} |x_d^{-1}u_k(z)|^p x_d^{\gamma + p (1-\alpha)} dz \\
& \leq N \|Du_k\|^p_{L_\infty(C_1(z_0))} \\
& \leq N\Big[ \|Du_k\|^p_{L_{p}(C_2(z_0), x_d^\gamma\, dz)} + \|\chi^{-\alpha/2} u_k\|^p_{L_{p}(C_2(z_0), x_d^\gamma\, dz)}\Big].
\end{split}
\]
Then, by taking the sum of this inequality for $(t_0, x_0') \in \mathcal{I}$, we also obtain
\[
\|\chi^{-\alpha} u_k\|_{L_p(\Omega_T' \times (0,1), x_d^\gamma\, dz)} \leq N \Big[ \|Du_k\|_{L_p(\Omega_T, x_d^\gamma\, dz)} + \|\chi^{-\alpha}u_k\|_{L_p(\Omega_T, x_d^\gamma\, dz)} \Big].
\]
Similarly, we also have $\chi^{-\alpha} (u_k)_t, Du_k\in L_p(\Omega_T' \times (0,1), x_d^\gamma\,dz)$.
By using the different quotient, we also get $DD_{x'} u_k \in L_p(\Omega_T' \times (0,1), x_d^\gamma\,dz)$.
From this, and the PDE of $u_k$, we have $D^2u_k \in L_p(\Omega_T' \times (0,1), x_d^\gamma\, dz)$.
Therefore, \eqref{near-est-uk} holds. The proof of the claim in this case is completed.

\smallskip \noindent
{\em Case} 2.
We consider $\gamma =p-1$. Let $\{f_k\}_k$ be as in \eqref{fk-approx} and let $\bar{\gamma} \in (p(\alpha-1)_+ -1, p-1)$.  As in {\em Case \textup{1}}, we can find a weak solution $u_k \in \cH^1_p(\Omega_T, x_d^{\bar{\gamma}}\, dz)$ to the divergence form equation \eqref{eq:xd-div} with $f_k$ in place of $f$, and
\begin{equation} \label{uk-Hp-est-b}
\|Du_k\|_{L_p(\Omega_T, x_d^{\bar{\gamma}}\, dz)} +  \|\chi^{-\alpha/2}u_k\|_{L_p(\Omega_T, x_d^{\bar{\gamma}}\, dz)} < \infty.
\end{equation}
We claim that for each $k \in \mathbb{N}$,
\begin{equation} \label{uk-Hp-est-b-1}
\|Du_k\|_{L_p(\Omega_T, x_d^{\gamma}\, dz)} +  \|\chi^{-\alpha/2}u_k\|_{L_p(\Omega_T, x_d^\gamma\, dz)} < \infty.
\end{equation}
Once this claim is proved, we can follow the proof in {\em Case \textup{1}} to obtain the existence of a solution $u \in \sW^{1,2}_p(\Omega_T, x_d^\gamma\, dz)$. Therefore, we only need to prove \eqref{uk-Hp-est-b-1}.

Let us fix $k \in \mathbb{N}$ and let $0 < r_0 < R_0$ such that \eqref{fk-support} holds. As $\bar{\gamma} < \gamma$, we see that
\[
\begin{split}
& \|Du_k\|_{L_p(\Omega_T' \times (0, 2R_0), x_d^{\gamma}\, dz)} +  \|\chi^{-\alpha/2}u_k\|_{L_p(\Omega_T' \times (0, 2R_0), x_d^\gamma\, dz)} \\
& \leq N\Big[ \|Du_k\|_{L_p(\Omega_T' \times (0, 2R_0), x_d^{\bar\gamma}\, dz)} +  \|\chi^{-\alpha/2}u_k\|_{L_p(\Omega_T' \times (0, 2R_0), x_d^{\bar\gamma}\, dz)}\Big] <\infty
\end{split}
\]
due to \eqref{uk-Hp-est-b}. Hence, it remains to prove
\begin{equation} \label{uk-Hp-est-b-2}
\|Du_k\|_{L_p(\Omega_T' \times (2R_0, \infty), x_d^{\gamma}\, dz)} +  \|\chi^{-\alpha/2}u_k\|_{L_p(\Omega_T' \times (2R_0, \infty), x_d^\gamma\, dz)} < \infty.
\end{equation}
To prove \eqref{uk-Hp-est-b-2}, we use the localization technique along the $x_d$ variable.
See \cite[Proof of Theorem 4.5, Case II]{DP-JFA}. We skip the details.
\end{proof}

\section{Equations with partially VMO coefficients} \label{sec:4}
We study \eqref{eq:main} in this section.
Precisely, we consider the equation
\begin{equation}\label{eq:main-1}
\begin{cases}
\sL u=\mu(x_d) f \quad &\text{ in } \Omega_T,\\
u=0 \quad &\text{ on } (-\infty, T) \times \partial \bR^d_+,
\end{cases}
\end{equation}
where $\sL$ is defined in \eqref{L-def} in which the coefficients $a_0$, $c_0$, and $a_{ij}$ are measurable functions depending on $z = (z', x_d) \in \Omega_T$.  We employ the perturbation method by freezing the coefficients. For $z_0 = (z'_0, x_{0d}) \in \overline{\Omega}_T$, let $[{a}_{ij}]_{Q_{\rho}'(z'_0)}, [a_{0}]_{Q_{\rho}'(z'_0)}$, and $[c_{0}]_{Q_{\rho}'(z'_0)}$ be functions defined in Assumption \ref{assumption:osc} $(\delta, \gamma_1, \rho_0)$, and we denote
\begin{equation} \label{a-sharp-def}
\begin{split}
a^{\#}_{\rho_0}(z_0) & =\sup_{\rho\in(0,\rho_0)}\left[ \max_{i,j=1, 2,\ldots, d}\fint_{Q_{\rho}^+(z_0)}|a_{ij}(z) -[{a}_{ij}]_{Q_{\rho}'(z'_0)}(x_d)| \mu_1(dz)  \right. \\
& \qquad +   \fint_{Q_{\rho}^+(z_0)}|a_{0}(z) -[{a}_{0}]_{Q_{\rho}'(z'_0)}(x_d)| \mu_1(dz) \\
& \qquad \left. +  \fint_{Q_{\rho}^+(z_0)}|c_{0}(z) -[{c}_{0}]_{Q_{\rho}'(z'_0)}(x_d)| \mu_1(dz) \right].
\end{split}
\end{equation}
For the reader's convenience, recall that $\mu_1, \bar{\mu}_1$ are defined in \eqref{mu-1-def}. We also recall that for a given $u$, we denote
\[
U = (\chi^{-\beta_0} u_t,  \chi^{\alpha-\beta_0}DD_{x'} u,  \lambda \chi^{-\beta_0}u).
\]
We also denote
\[
U' = (\chi^{-\beta_0} u_t,  \chi^{\alpha-\beta_0}DD_{x'} u, \lambda^{1/2} \chi^{\alpha/2-\beta_0}Du, \lambda \chi^{-\beta_0}u).
\]
We begin with the following oscillation estimates for solutions to \eqref{eq:main-1} that have small supports in the time-variable.
\begin{lemma} \label{osc-est-small} Let $\nu, \rho_0 \in (0,1)$, $p_0 \in (1, \infty)$,  $\alpha \in (0,2)$, $\beta_0 \in {(\alpha-1}, \min\{1, \alpha\}]$, $\gamma_1 \in (p_0(\beta_0-\alpha +1) -1, p_0(\beta_0-\alpha +2)-1)$, and $\gamma_1' = \gamma_1-p_0(\beta_0-\alpha)\in (p_0-1,2p_0-1)$. Assume that $u \in \sW^{1,2}_{p}(Q_{6\rho}^+(z_0), x_d^{\gamma_1'}dz)$ is a strong solution of
\[
\left\{
\begin{array}{cccl}
\sL u & = & \mu(x_d) f & \quad \text{in} \quad Q_{6\rho}^+(z_0),\\
u & = & 0 & \quad \text{on} \quad  Q_{6\rho}(z_0) \cap \{x_d =0\}
\end{array} \right.
\]
for $f \in L_{p_0}(Q_{6\rho}^+(z_0), x_d^{{\gamma_1'}}dz)$, and $\textup{supp}(u) \subset (t_1 -(\rho_0 \rho_1)^{2-\alpha}, t_1 +(\rho_0 \rho_1)^{2-\alpha})$ for some $t_1 \in \bR$ and $\rho_0 >0$. Then,
\begin{align} \notag
&\big (|U - (U)_{Q_{\kappa\rho}^+(z_0), \mu_1}|\big)_{Q_{\kappa\rho}^+(z_0), \mu_1} \\ \notag
& \leq  N \Big[\kappa^{\theta}  + \kappa^{-(d + (\gamma_1)_+ +2-\alpha)/p_0} \big( a_{\rho_0}^{\#}(z_0)^{\frac{1}{p_0} - \frac{1}{p}} + \rho_1^{(2-\alpha)(1-1/p_0)}\big) \Big]   (|U|^{p})_{Q_{14 \rho}^+(z_0), \mu_1}^{1/p}  \\ \label{U-osc-gen}
& \qquad + N \kappa^{-(d + (\gamma_1)_+ +2-\alpha)/p_0} (|\chi^{\alpha-\beta_0} f|^{p_0})_{Q_{14\rho}^+(z_0), \mu_1}^{1/p_0},
\end{align}
where $\theta>0$ is defined in Lemma \ref{oscil-lemma-1},   $p\in (p_0,\infty)$, and $N = N(p, p_0, \gamma_1, \alpha, \beta_0, d, \nu)>0$. In addition, if $\beta_0 \in {(\alpha-1}, \alpha/2]$, we also have
\begin{align} \notag
&\big (|U' - (U')_{Q_{\kappa\rho}^+(z_0), \mu_1}|\big)_{Q_{\kappa\rho}^+(z_0), \mu_1} \\ \notag
& \leq  N \Big[\kappa^{\theta}  + \kappa^{-(d + (\gamma_1)_+ +2-\alpha)/p_0} \big( a_{\rho_0}^{\#}(z_0)^{\frac{1}{p_0} - \frac{1}{p}} + \rho_1^{(2-\alpha)(1-1/p_0)}\big) \Big]   (|U'|^{p})_{Q_{14 \rho}^+(z_0), \mu_1}^{1/p}  \\ \label{Uall-osc-gen}
& \qquad + N \kappa^{-(d + (\gamma_1)_+ +2-\alpha)/p_0} (|\chi^{\alpha-\beta_0} f|^{p_0})_{Q_{14\rho}^+(z_0), \mu_1}^{1/p_0}.
\end{align}
 \end{lemma}
\begin{proof} We split the proof into two cases.

\smallskip
\noindent
{\em Case \textup{1}.} We consider $\rho < \rho_0/14$. We denote
\[
\sL_{\rho, z_0} u =[a_0]_{Q_{6\rho}'(z_0')}(x_d) u_t + \lambda [c_0]_{Q_{6\rho}'(z_0')}(x_d)u - \mu(x_d) [a_{ij}]_{Q_{6\rho}'(z_0')}(x_d) D_i D_j u
\]
and
\[
\begin{split}
\tilde{f}(z) & = f(z) + [a_{ij} - [a_{ij}]_{Q_{6\rho}'(z_0')}(x_d)] D_i D_ju \\
& \qquad +  \big[ \lambda ([c_0]_{Q_{6\rho}'(z_0')} - c_0) u + ([a_0]_{Q_{6\rho}'(z_0')} - a_0) u_t \big]/\mu(x_d).
\end{split}
\]
Then, $u \in \sW^{1,2}_{p}(Q_{6\rho}^+(z_0), x_d^{\gamma_1'}dz)$ is a strong solution of
\[
\left\{
\begin{array}{cccl}
\sL_{\rho, z_0} u  & = &\mu(x_d) \tilde{f}  & \quad \text{in} \quad Q_{6\rho}^+(z_0)\\
u & = & 0 & \quad \text{on} \quad Q_{6\rho}^+(z_0) \cap \{x_d =0\}.
\end{array} \right.
\]
We note that due to \eqref{add-assumption}, the term $a_{dd} - \bar{a}_{dd} =0$.
Therefore, by using H\"{o}lder's inequality and \eqref{con:ellipticity}, we obtain
\[
\begin{split}
& \left(\fint_{Q_{14\rho}^+(z_0)}|\chi^{\alpha-\beta_0} \big(a_{ij} - [a_{ij}]_{Q_{6\rho}'(z_0')}(x_d)\big) D_i D_ju|^{p_0} \mu_1(dz) \right)^{1/p_0} \\
 & \leq \left(\fint_{Q_{14\rho}^+(z_0)}|a_{ij} - [a_{ij}]_{Q_{6\rho}'(z_0')}(x_d)|^{pp_0/(p-p_0)}  \mu_1(dz) \right)^{\frac{1}{p_0} -\frac{1}{p}} \\
 & \qquad \times \left(\fint_{Q_{14\rho}^+(z_0)}|\chi^{\alpha-\beta_0} DD_{x'}u|^{p} \mu_1(dz)\right)^{1/p} \\
& \leq N a_{\rho_0}^{\#}(z_0)^{\frac{1}{p_0} - \frac{1}{p}} \left(\fint_{Q_{14\rho}^+(z_0)}|\chi^{\alpha-\beta_0} DD_{x'}u|^{p} \mu_1(dz)\right)^{1/p}.
 \end{split}
\]
By a similar calculation using \eqref{con:mu}, we also obtain the estimate for the term $\big[ \lambda ([c_0]_{Q_{6\rho}'(z_0')}(x_d)  - c_0) u + ([a_{0}]_{Q_{6\rho}'(z_0')}(x_d) - a_0) u_t \big]/\mu(x_d)$. Thus,
\[
\begin{split}
 (|\chi^{\alpha-\beta_0} \tilde{f}|^{p_0})_{Q_{14\rho}^+(z_0), \mu_1}^{1/p_0} & \leq  (|\chi^{\alpha-\beta_0} f|^{p_0})_{Q_{14\rho}^+(z_0), \mu_1}^{1/p_0} \\
 & \qquad +  N a_{\rho_0}^{\#}(z_0)^{\frac{1}{p_0} - \frac{1}{p}} (|U|^{p})_{Q_{14\rho}^+(z_0), \mu_1}^{1/p}.
 \end{split}
\]
Then, applying Lemma \ref{oscil-lemma-2}, we obtain
\[
\begin{split}
& (|U - (U)_{Q_{\kappa \rho}^+(z_0), \mu_1}|)_{Q_{\kappa \rho}^+(z_0), \mu_1}\\
&\leq N  \kappa^{\theta} (|U|^{p_0})_{Q_{14 \rho}^+(z_0), \mu_1}^{1/p_0} + N \kappa^{-(d + (\gamma_1)_+ +2-\alpha)/p_0}  (|\chi^{\alpha-\beta_0} \tilde{f}|^{p_0})_{Q_{14\rho}^+(z_0), \mu_1}^{1/p_0} \\
& \leq N \big(\kappa^{\theta}  + \kappa^{-(d + (\gamma_1)_+ +2-\alpha)/p_0} a_{\rho_0}^{\#}(z_0)^{\frac{1}{p_0} - \frac{1}{p}} \big)   (|U|^{p})_{Q_{14 \rho}^+(z_0), \mu_1}^{1/p}  \\
& \qquad + N \kappa^{-(d + (\gamma_1)_+ +2-\alpha)/p_0} (|\chi^{\alpha-\beta_0} f|^{p_0})_{Q_{14\rho}^+(z_0), \mu_1}^{1/p_0}.
\end{split}
\]
Therefore, \eqref{U-osc-gen} holds. In a similar way but  applying  Remark \ref{all-oscilla-est}, we also obtain \eqref{Uall-osc-gen}.

\smallskip
\noindent
{\em Case 2.} Consider $\rho \geq \rho_0/14$.  Denoting $\Gamma = (t_1 -(\rho_0 \rho_1)^{2-\alpha}, t_1 + (\rho_0 \rho_1)^{2-\alpha})$, we apply \eqref{Q-compared} and the triangle inequality to infer that
\[
\begin{split}
& \fint_{Q_{\kappa\rho}^+(z_0)} |U - (U)_{Q_{\kappa\rho}^+(z_0), \mu_1}|
\mu_1(dz)\leq 2 \fint_{Q_{\kappa\rho}^+(z_0)} |U(z)|\mu_1(dz) \\
& \leq N \kappa^{-(d+2 -\alpha + (\gamma_1)_+)} \left(\fint_{Q_{14\rho}^+(z_0)} |U(z)|^{p_0}\mu_1(dz)\right)^{\frac 1 {p_0}} \left(\fint_{Q_{14\rho}^+(z_0)} \mathbf{1}_{\Gamma}(z) \mu_1(dz)\right)^{1-\frac 1 {p_0}} \\
& \leq N \kappa^{-(d+2 -\alpha + (\gamma_1)_+)} \rho_1^{(2-\alpha)(1-1/p_0)}\left(\fint_{Q_{14\rho}^+(z_0)} |U(z)|^{p_0}\mu_0(dz)\right)^{1/p_0} \\
& \leq N \kappa^{-(d+2 -\alpha + (\gamma_1)_+)} \rho_1^{(2-\alpha)(1-1/p_0)}(|U|^{p})_{Q_{14 \rho}^+(z_0), \mu_1}^{1/p} .
\end{split}
\]
Therefore,  \eqref{U-osc-gen} follows.  Similarly, \eqref{Uall-osc-gen} can be proved.
\end{proof}
Our next lemma gives the a priori estimates of solutions having small supports in $t$.
\begin{lemma}[Estimates of solutions having small supports] \label{small-support-sol} Let $T \in (-\infty, \infty]$, $\nu \in (0,1)$, $p, q, K \in (1, \infty)$, $\alpha \in (0, 2)$, and $\gamma_1 \in (\beta_0 -\alpha, \beta_0 -\alpha +1]$ for $\beta_0 \in {(\alpha-1}, \min\{1, \alpha\}]$. Then, there exist sufficiently small positive numbers $\delta$ and  $\rho_1$, depending on $d, \nu, p, q, K, \alpha{,\beta_0}$, and $\gamma_1$, such that the following assertion holds. Suppose that $\omega_0 \in A_q(\bR)$, $\omega_1 \in A_p(\bR^d_+, \mu_1)$ with
\[
[\omega_0]_{A_q(\bR)} \leq K \qquad \text{and} \qquad [\omega_1]_{A_p(\bR^d_+, \mu_1)} \leq K.
\]
Suppose that \eqref{con:mu}, \eqref{con:ellipticity}, and \eqref{add-assumption} hold, and \textup{Assumption \ref{assumption:osc}}$(\delta, \gamma_1, \rho_0)$ holds with some $\rho_0>0$. If $u \in \sW^{1,2}_{q,p}(\Omega_T, x_d^{p(\alpha-\beta_0)} \omega\, d\mu_1)$ is a strong solution to \eqref{eq:main} with some $\lambda>0$ and a function $f\in L_{q,p}(\Omega_T, x_d^{p(\alpha-\beta_0)} \omega\, d\mu_1)$, and $u$ vanishes outside $(t_1 - (\rho_0\rho_1)^{2-\alpha}, t_1+(\rho_0\rho_1)^{2-\alpha})$ for some $t_1 \in \bR$, then
\begin{equation} \label{est-1-small-supp}
\|\chi^{-\alpha} u_t\|_{L_{q,p}} + \|D^2u\|_{L_{q,p}} + \lambda \|\chi^{-\alpha} u\|_{L_{q,p}} \leq N \|f\|_{L_{q,p}},
\end{equation}
where  $N = N(d,\nu, p, q, \alpha,{ \beta_0,}\gamma_1, K)>0$, $L_{q,p}=L_{q,p}(\Omega_T, x_d^{p(\alpha-\beta_0)} \omega\, d\mu_1)$, $\omega(t,x) =\omega_0(t)\omega_1(x)$ for $(t,x) \in \Omega_T$, and $\mu_1(dz) = x_d^{\gamma_1}\, dxdt$. Moreover, if $\beta_0 \in [0, \alpha/2]$, then it also holds that
\begin{equation} \label{est-2-small-supp}
\begin{split}
& \|\chi^{-\alpha} u_t\|_{L_{q,p}} + \|D^2u\|_{L_{q,p}} + \lambda^{1/2} \|\chi^{-\alpha/2} Du\|_{L_{q,p}}  + \lambda \|\chi^{-\alpha} u\|_{L_{q,p}} \\
& \leq N \|f\|_{L_{q,p}}.
\end{split}
\end{equation}
\end{lemma}
\begin{proof}
As $\omega_0 \in A_q((-\infty,T))$ and $\omega_1  \in A_p(\bR^d_+, d\mu_1)$, by  the reverse H\"older's inequality \cite[Theorem 3.2]{MS1981}, we find $p_1=p_1(d,p,q,\gamma_1,K)\in (1,\min(p,q))$ such that
\begin{equation} \label{eq0605_13}
\omega_0 \in A_{q/p_1}((-\infty,T)),\quad
\omega_1  \in A_{p/p_1}(\bR^d_+, \mu_1).
\end{equation}
Because $\gamma_1 \in (\beta_0 -\alpha, \beta_0 -\alpha +1]$, we can choose $p_0 \in (1, p_1)$ sufficiently closed to $1$ so that
\[
\gamma_1 \in (p_0(\beta_0-\alpha +1) -1, p_0(\beta_0-\alpha +2) -1).
\]
By \eqref{U-osc-gen} of Lemma \ref{osc-est-small}  and H\"{o}lder's inequality,  we have
\[
\begin{split}
U^{\#}_{\mu_1} \leq &  N \Big[\kappa^{\theta}  + \kappa^{-(d + (\gamma_1)_+ +2-\alpha)/p_0} \big(a_{\rho_0}^{\#}(z_0)^{\frac{1}{p_0} - \frac{1}{p_1}} + \rho_1^{(2-\alpha)(1-1/p_0)}\big) \Big]   \cM_{\mu_1}(|U|^{p_1})^{1/p_1}  \\
& \qquad + N \kappa^{-(d + (\gamma_1)_+ +2-\alpha)/p_0} \cM_{\mu_1}(|\chi^{\alpha-\beta_0} f|^{p_1})^{1/p_1}
 \quad \text{in} \quad \overline{\Omega_T}
\end{split}
\]
 for any $\kappa\in (0,1)$, where  $N = N(\nu, d, p_0, p_1, \alpha,\beta_0, \gamma_1) >0$ and $a_{\rho_0}^{\#}$ is defined in \eqref{a-sharp-def}. Therefore, it follows from Theorem \ref{FS-thm} and \eqref{eq0605_13} that
\[
\begin{split}
& \norm{U}_{L_{q,p}(\Omega_T, \omega\, d\mu_1)}\\
&  \leq  N \Big[\kappa^{\theta}  + \kappa^{-(d + (\gamma_1)_+ +2-\alpha)/p_0} \big(\delta^{\frac{1}{p_0} - \frac{1}{p_1}} + \rho_1^{(2-\alpha)(1-1/p_0)}\big) \Big] \times \\
& \qquad \qquad \times \|\cM_{\mu_1}(|U|^{p_1})^{1/p_1}\|_{L_{q,p}(\Omega_T, \omega\, d\mu_1)}   \\
&  \qquad +  N \kappa^{-(d + (\gamma_1)_+ +2-\alpha)/p_0} \| \cM_{\mu_1} (|\chi^{\alpha -\beta_0} f|^{p_1})^{\frac 1 {p_1}}\|_{L_{q,p}(\Omega_T, \omega\, d\mu_1)} \\
& \leq  N \Big[\kappa^{\theta}  + \kappa^{-(d + (\gamma_1)_+ +2-\alpha)/p_0} \big(\delta^{\frac{1}{p_0} - \frac{1}{p_1}} + \rho_1^{(2-\alpha)(1-1/p_0)}\big) \Big] \|U\|_{L_{q,p}(\Omega_T, \omega\, d\mu_1)}  \\
& \qquad +  N \kappa^{-(d + (\gamma_1)_+ +2-\alpha)/p_0} \|\chi^{\alpha-\beta_0} f\|_{L_{q,p}(\Omega_T, \omega\, d\mu_1)},
\end{split}
\]
where $N = N(d,\nu, p, q, \alpha,\beta_0, \gamma_1, K)>0$.   Now, by choosing $\kappa$ sufficiently small and then $\delta$ and $\rho_1$ sufficiently small depending on $d,\nu, p, q,\alpha, \gamma_1 , \beta_0$, and $K$ such that
\[
N\Big [\kappa^{\theta}  + \kappa^{-(d + (\gamma_1)_+ +2-\alpha)/p_0} \big(\delta^{\frac{1}{p_0} - \frac{1}{p_1}} + \rho_1^{(2-\alpha)(1-1/p_0)}\big)  \Big] <1/2,
\]
we obtain
\[
\begin{split}
& \norm{U}_{L_{q,p}(\Omega_T, \omega\, d\mu_1)}   \leq  N(d, \nu, p, q, \alpha,\beta_0, \gamma_1, K)  \|\chi^{\alpha -\beta_0} f\|_{L_{q,p}(\Omega_T, \omega\, d\mu_1)}.
\end{split}
\]
From this and the PDE in \eqref{eq:main}, we obtain
\[
\|\chi^{-\alpha} u_t\|_{L_{q,p}} + \|D^2u\|_{L_{q,p}} + \lambda \|\chi^{-\alpha} u\|_{L_{q,p}} \leq N \| f\|_{L_{q,p}}.
\]
This proves \eqref{est-1-small-supp}. The proof of  \eqref{est-2-small-supp} is similar by applying \eqref{Uall-osc-gen} instead of \eqref{U-osc-gen}.
\end{proof}

Below, we provide the proof of Theorem \ref{main-thrm}.

\begin{lemma}[A priori estimates of solutions]  \label{apriori-est-lemma} Let $T \in (-\infty, \infty]$, $ \nu \in (0,1)$, $p, q, K \in (1, \infty)$, $\alpha \in (0, 2)$, and $\gamma_1 \in (\beta_0 -\alpha, \beta_0 -\alpha +1]$ for $\beta_0 \in {(\alpha-1}, \min\{1, \alpha\}]$. Then, there exist $\delta = \delta(d, \nu, p, q, K, \alpha, \beta_0,\gamma_1)>0$ sufficiently small and $\lambda_0 = \lambda_0(d, \nu, p, q, K, \alpha,\beta_0, \gamma_1)>0$ sufficiently large such that the following assertions hold. Let $\omega_0 \in A_q(\bR)$, $\omega_1 \in A_p(\bR^d_+, \mu_1)$ satisfy
\[
[\omega_0]_{A_q(\bR)} \leq K \qquad \text{and} \qquad [\omega_1]_{A_p(\bR^d_+, \mu_1)} \leq K.
\]
Suppose that \eqref{con:mu}, \eqref{con:ellipticity}, and \eqref{add-assumption} hold, and suppose that \textup{Assumption \ref{assumption:osc}}$ (\delta, \gamma_1, \rho_0)$ holds with some $\rho_0>0$. If $u \in \sW^{1,2}_{q,p}(\Omega_T, x_d^{p(\alpha-\beta_0)}\omega\,d\mu_1)$  is a strong solution to \eqref{eq:main} with some $\lambda{\ge \lambda_0\rho_0^{\alpha-2}}$ and  $f \in L_{q,p}(\Omega_T, x_d^{p(\alpha-\beta_0)}\omega\, d\mu_1)$, then
\begin{equation} \label{main-est-1-b}
\|\chi^{-\alpha} u_t\|_{L_{q,p}} + \|D^2u\|_{L_{q,p}} + \lambda \|\chi^{-\alpha} u\|_{L_{q,p}} \leq N \|f\|_{L_{q,p}},
\end{equation}
where $\omega(t, x) = \omega_0(t) \omega_1(x)$ for $(t,x) \in \Omega_T$, $L_{q,p} = L_{q,p}(\Omega_T, x_d^{p(\alpha-\beta_0)}\omega \,d\mu_1)$, and $N = N(d, \nu, p, q, K, \alpha, \beta_0, \gamma_1)>0$. Moreover, if $\beta_0 \in {(\alpha-1}, \alpha/2]$, then it also holds that
\begin{equation} \label{main-est-2-b}
\begin{split}
& \|\chi^{-\alpha} u_t\|_{L_{q,p}} + \|D^2u\|_{L_{q,p}} + \lambda^{1/2} \|\chi^{-\alpha/2} Du\|_{L_{q,p}}  + \lambda \|\chi^{-\alpha} u\|_{L_{q,p}} \\
& \leq N \|f\|_{L_{q,p}}.
\end{split}
\end{equation}
\end{lemma}
\begin{proof}  Let $\delta, \rho_1$ be positive numbers defined in Lemma \ref{small-support-sol}, and let  $\lambda_0>$ be a number sufficiently large to be determined, depending on $d, p, q, \alpha,{\beta_0,\nu,} \gamma_1, K$. As the proof of \eqref{main-est-1-b} and of \eqref{main-est-2-b} are similar, we only prove the a priori estimate \eqref{main-est-1-b}.  We use a partition of unity argument in the time variable.
Let $\delta>0$ and $\rho_1>0$ be as in Lemma \ref{small-support-sol} and let
$$
\xi=\xi(t) \in C_0^\infty( -(\rho_0\rho_1)^{2-\alpha}, (\rho_0\rho_1)^{2-\alpha})
$$
be a  non-negative cut-off function satisfying
\begin{equation} \label{xi-0702}
\int_{\bR} \xi(s)^q\, ds =1 \qquad \text{and} \qquad  \int_{\bR}|\xi'(s)|^q\,ds \leq \frac{N}{(\rho_0\rho_1)^{q(2-\alpha)}}
\end{equation}
for some $N = N(q)>0$.
For fixed $s \in (-\infty,  \infty)$, let $u^{(s)}(z) = u(z) \xi(t-s)$ for $z = (t, x) \in \Omega_T$.
We see that $u^{(s)} \in \sW^{1,2}_p(\Omega_T,x_d^{p(\alpha-\beta_0)}\omega\, d\mu_1)$ is a strong solution of
\[
\sL u^{(s)}(z)  =\mu(x_d)  f^{(s)} (z) \quad \text{in} \quad \Omega_T
\]
with the boundary condition $u^{(s)} =0$ on $\{x_d =0\}$, where
\[
f^{(s)}(z)   = \xi(t-s) f(z)  +  \xi'(t-s) u(z)/\mu(x_d).
\]
As $\text{spt}(u^{(s)}) \subset (s -(\rho_0\rho_1)^{2-\alpha}, s+ (\rho_0\rho_1)^{2-\alpha}) \times \bR^{d}_{+}$, we apply Lemma \ref{small-support-sol} to get
\[
\begin{split}
\|\chi^{-\alpha} \partial_tu^{(s)}\|_{L_{q,p}} + \|D^2u^{(s)}\|_{L_{q,p}} + \lambda \|\chi^{-\alpha} u^{(s)}\|_{L_{q,p}} \leq N \|f^{(s)}\|_{L_{q,p}}.
\end{split}
\]
Then, by integrating the $q$-th power of this estimate with respect to $s$, we get
\begin{align}\notag
& \int_{\bR}\Big(|\chi^{-\alpha} \partial_tu^{(s)}\|_{L_{q,p}}^q + \|D^2u^{(s)}\|_{L_{q,p}}^q + \lambda^q \|\chi^{-\alpha} u^{(s)}\|_{L_{q,p}}^q\Big)\, ds\\  \label{par-int-0515}
&  \leq N\int_{\bR} \|f^{(s)}\|_{L_{q,p}}^q\, ds
\end{align}
where  $N = N(d,\nu, p, q, K, \alpha,{ \beta_0,}\gamma_1)>0$. Now, by the Fubini theorem and \eqref{xi-0702}, it follows that
\[
\begin{split}
& \int_{\bR}\|\chi^{-\alpha} \partial_tu^{(s)}\|_{L_{q,p}}^q\, ds\\
& = \int_{\bR} \left(\int_{-\infty}^T \|\chi^{-\alpha}u_t(t,\cdot)\|_{L_p(\bR^d_+, x_d^{p(\alpha-\beta_0)} \omega_1\, d\mu_1)}^q \omega_0(t) \xi^q(t-s)\, dt \right)\, ds \\
&= \int_{-\infty}^T  \left( \int_{\bR}\xi^q(t-s)\, ds \right) \|\chi^{-\alpha}u_t(t,\cdot)\|_{L_p(\bR^d_+,x_d^{p(\alpha-\beta_0)} \omega_1\, d\mu_1)}^q \omega_0(t)\, dt \\
& = \|\chi^{-\alpha}u_t \|_{L_{q,p}(\bR^d_+,x_d^{p(\alpha-\beta_0)} \omega\, d\mu_1)}^q,
\end{split}
\]
and similarly
\[
\begin{split}
& \int_{\bR} \| D^2u^{(s)}\|_{L_{q,p}}^q\, ds =  \|\chi^{\alpha-\beta_0} D^2u\|_{L_{q,p}}^q, \\
& \int_{\bR} \|\chi^{-\alpha} u^{(s)}\|_{L_{q,p}}^q\, ds = \|\chi^{-\beta_0} u\|_{L_{q,p}}^q .
\end{split}
\]
Moreover,
\[
\int_{\bR} \|f^{(s)}\|_{L_{q,p}}^q\, ds \leq \|f\|_{L_{q,p}}^q + \frac{N}{(\rho_0\rho_1)^{q(2-\alpha)}}   \|\chi^{-\alpha}u\|_{L_{q,p}}^q,
\]
where \eqref{xi-0702} is used and $N = N(q)>0$. As $\rho_1$ depends on $d, \nu, p, q, K, \alpha{,\beta_0,\gamma_1}$, by combining the estimates we just derived, we infer from \eqref{par-int-0515} that
\[
\begin{split}
&  \|\chi^{-\alpha} \partial_tu \|_{L_{q,p}} + \| D^2u\|_{L_{q,p}} + \lambda \|\chi^{-\alpha} u\|_{L_{q,p}} \leq N\Big(\|f\|_{L_{q,p}}  +  \rho_0 ^{\alpha-2}   \|\chi^{-\alpha}u\|_{L_{q,p}} \Big)
\end{split}
\]
with $N = N(d,\nu, p, q, K, \alpha,{ \beta_0,}\gamma_1)>0$.
Now, we choose $\lambda_0 = 2N$. Then for $\lambda \geq \lambda_0 \rho_0^{\alpha-2}$, we have
\[
\begin{split}
&  \|\chi^{-\alpha} \partial_tu \|_{L_{q,p}} + \|D^2u\|_{L_{q,p}} + \lambda \|\chi^{-\alpha} u\|_{L_{q,p}} \leq N \| f\|_{L_{q,p}}  .
\end{split}
\]
This estimate yields \eqref{main-est-1-b}.
\end{proof}

\smallskip
Now, we have all ingredients to complete the proof of Theorem \ref{main-thrm}.
\begin{proof}[Proof of Theorem \ref{main-thrm}] The a priori estimates \eqref{main-est-1} and \eqref{main-est-2} follow from Lemma \ref{apriori-est-lemma}. Hence, it remains to prove the existence of solutions.  We employ the the technique introduced in \cite[Section 8]{Dong-Kim-18}. See also \cite[Proof of Theorem 2.3]{DP-JFA}. The proof is split into two steps, and we only outline the key ideas in each step.

\smallskip
\noindent
{\em Step \textup{1}.} We consider the case $p =q$, $\omega_0 \equiv 1$, and $\omega_1 \equiv 1$. We employ the method of continuity. Consider the operator
\[
\sL_\tau   = (1-\tau)\big(\partial_t + \lambda  - \mu(x_d) \Delta\big) + \tau \sL, \qquad \tau \in [0, 1].
\]
It is a simple calculation to check that the assumptions in Theorem \ref{main-thrm} are satisfied uniformly with respect to $\tau \in [0,1]$. Then, using the solvability in Theorem \ref{thm:xd} and the a priori estimates obtained in Lemma \ref{apriori-est-lemma}, we get the existence of a solution $u \in \sW^{1,2}_p(\Omega_T, x_d^{p(\alpha-\beta_0)}\, d\mu_1)$ to \eqref{eq:main} when $\lambda \geq \lambda_0 \rho_0^{\alpha-2}$, where $\lambda_0$ is the constant  in Lemma \ref{apriori-est-lemma}.

\smallskip
\noindent
{\em Step \textup{2}.} We combine {\em Step \textup{1}} and Lemma \ref{apriori-est-lemma} to prove the existence of {a strong} solution $u$ satisfying \eqref{main-est-1}.
Let $p_1 > \max\{p,q\}$ be sufficiently large and let $\varepsilon_1, \varepsilon_2 \in (0,1)$ be sufficiently small depending on $K, p, q$, and $\gamma_1$ such that
\begin{equation} \label{epsilon12-def}
1-\frac{p}{p_1} = \frac{1}{1+\varepsilon_1} \qquad \text{and} \qquad 1 - \frac{q}{p_1} = \frac{1}{1+\varepsilon_2},
\end{equation}
and both $\omega_1^{1+\varepsilon_1}$ and $\omega_0^{1+\varepsilon_2}$ are locally integrable and satisfy the doubling property. Specifically, there is $N_0>0$ such that
\begin{equation} \label{omega-0}
\int_{\Gamma_{2r}(t_0)} \omega_0^{1+\varepsilon_2}(s)\, ds \leq N_0 \int_{\Gamma_{r}(t_0)} \omega_0^{1+\varepsilon_2}(s)\, ds
\end{equation}
for any $r>0$ and $t_0 \in \bR$, where $\Gamma_{r}(t_0) = (t_0 -r^{2-\alpha}, \min\{t_0 + r^{2-\alpha}, T\})$. Similarly
\begin{equation} \label{omega-1-0308}
\int_{B_{2r}^+(x_0)} \omega_1^{1+\varepsilon_1}(x)\, d\mu_1 \leq N_0\int_{B_{r}^+(x_0)} \omega_1^{1+\varepsilon_1}(x)\, d\mu_1
\end{equation}
for any $r >0$ and $x_0 \in \overline{\bR^d_+}$.

\smallskip
Next, let $\{f_k\}$ be a sequence in $C_0^\infty(\Omega_T)$ such that
\begin{equation} \label{f-k-converge-0227}
\lim_{k\rightarrow \infty} \|f_k - f\|_{L_{q,p}(\Omega_T, x_d^{p(\alpha-\beta_0)}\omega\, d\mu_1)} =0.
\end{equation}
By {\em Step \textup{1}}, for each $k \in \mathbb{N}$, we can find a solution $u_k \in \sW^{1,2}_{p_1}(\Omega_T,x_d^{p_1(\alpha-\beta_0)}\, d\mu_1)$ of  \eqref{eq:main} with $f_k$ in place of $f$, where $\lambda \geq \lambda_0 \rho_0^{\alpha-2}$ for $\lambda_0 = \lambda_0(d, \nu, p_1, p_1, K, \alpha,\beta_0, \gamma_1)>0$. Observe that if the sequence $\{u_k\}$ is in $\sW^{1,2}_{q,p}(\Omega_T, x_d^{p(\alpha-\beta_0)}  \omega \, d\mu_1)$, then by applying the a priori estimates in Lemma \ref{apriori-est-lemma}, \eqref{f-k-converge-0227}, and the linearity of the equation \eqref{eq:main}, we conclude that  $\{u_k\}$ is Cauchy  in $\sW^{1,2}_{q,p}(\Omega_T, x_d^{p(\alpha-\beta_0)}  \omega\, d\mu_1)$. 
Let $u \in \sW^{1,2}_{q,p}(\Omega_T, x_d^{p(\alpha-\beta_0)}  \omega\, d\mu_1)$ be the limit of the sequence $\{u_k\}$. Then, by letting $k \rightarrow \infty$ in the equation for $u_k$, we see that $u$ solves \eqref{eq:main}.

\smallskip
It remains to prove that for each fixed $k \in \bN$, $u_k \in \sW^{1,2}_{q,p}(\Omega_T, x_d^{p(\alpha-\beta_0)}  \omega\, d\mu_1)$. To this end, let us denote
\[
D_{R} = (-R^{2-\alpha}, \min\{R^{2-\alpha}, T\}) \times B_R^+.\]
Then, let $R_0>0$ be sufficiently large such that
\begin{equation} \label{fk-spt}
\supp(f_k) \subset D_{R_0}.
\end{equation}
 We note that $R_0$ depends on $k$. It follows from \eqref{epsilon12-def}, \eqref{omega-0}, \eqref{omega-1-0308}, and  H\"{o}lder's inequality that
\[
\begin{split}
& \|u_k\|_{\sW^{1,2}_{q,p}(D_{2R_0},x_d^{p(\alpha-\beta_0)}  \omega d\mu_1)} \\
& \leq N(d, p, q, p_1, \alpha, \gamma_1, \beta_0, R_0) \|u_k\|_{\sW^{1,2}_{p_1}(D_{2R_0}, x_d^{p_1(\alpha-\beta_0)}  d\mu_1)} <\infty.
\end{split}
\]
Hence, we only need to prove
\[
\|u_k\|_{\sW^{1,2}_{q,p}(\Omega_T\setminus D_{R_0},x_d^{p(\alpha-\beta_0)}  \omega d\mu_1)}  <\infty.
\]
This is done by the localization technique employing \eqref{epsilon12-def}, \eqref{omega-0}, \eqref{omega-1-0308}, \eqref{fk-spt}, and  H\"{o}lder's inequality, using the fast decay property of solutions when the right-hand side is compactly supported.
We skip the details as the calculation is very similar to that of \cite[Section 8]{Dong-Kim-18}, and also of \cite[Step II - Proof of Theorem 2.3]{DP-JFA}. The proof of Theorem \ref{main-thrm} is completed.
\end{proof}
Next, we prove Corollary \ref{cor1}.
\begin{proof}[Proof of Corollary \ref{cor1}] It is sufficient to show that we can make the choices for $\gamma_1, \beta_0$, and $\omega_1$ to apply Theorem \ref{main-thrm} to obtain \eqref{cor-est-1} and \eqref{cor-est-2}. Indeed, the choices are similar to those in the proof of Theorem \ref{thm:xd}. To obtain \eqref{cor-est-1}, we take $\beta_0 = \min\{1, \alpha\}$, and with this choice of $\beta_0$, we have
\[
\alpha - \beta_0 = (\alpha-1)_+ \quad \text{and} \quad (\beta_0 -\alpha, \beta_0-\alpha +1] = (-(\alpha -1)_+, 1- (\alpha-1)_+].
\]
Then, let $\gamma_1 = 1- (\alpha-1)_+$ and $\gamma' = \gamma - [\gamma_1 + p(\alpha-1)_+]$.
From the choice of $\gamma_1$ and the condition on $\gamma$, we see that
\begin{equation}  \label{cond-gamma-1}
 -1-\gamma_1 < \gamma' < (1+\gamma_1) (p-1).
\end{equation}
Now, let
$ \omega_1(x) = x_d^{\gamma'}$ for $x \in \bR^d_+$.
It follows from \eqref{cond-gamma-1} that $\omega_1 \in A_p(\mu_1)$.
As Assumption $(\rho_0, \gamma_1, \delta)$ holds,
we can apply \eqref{main-est-1} to obtain \eqref{cor-est-1}.

\smallskip
Next, we prove \eqref{cor-est-2}. In this case, we choose $\beta_0 = \alpha/2$, $\gamma_1 = 1-\alpha/2$, and
\begin{equation} \label{cond-gamma-2}
  \gamma'   =\gamma - [\gamma_1 + p\alpha/2].
\end{equation}
We use the fact that $\gamma \in (p\alpha/2-1, 2p-1)$ and  \eqref{cond-gamma-2} to get \eqref{cond-gamma-1}.
As Assumption $(\rho_0, 1-\alpha/2, \delta)$ holds, by taking $\omega_1(x) = x_d^{\gamma'}$, we obtain \eqref{cor-est-2} from \eqref{main-est-2}.

\smallskip
Finally, we prove the last assertion of the corollary on the $C^{(1+\beta)/2, 1+\beta}$-regularity of the solution $u$. Note that the $C^{{(1+\beta)}/2, 1+\beta}$-regularity in the interior of $\Omega_T$ follows from the standard {parabolic} Sobolev embedding theorem. To prove the $C^{{(1+\beta)}/2, 1+\beta}$-regularity near $\{x_d=0\}$, it is sufficient to prove such regularity on $\overline{Q}_{1/2}^+$.
This follows immediately from  Proposition \ref{imbedding-prop}.
The proof  is completed.
\end{proof}

\section{Degenerate viscous Hamilton-Jacobi equations}\label{sec:5}
To demonstrate an application of the results in our paper, we consider the following degenerate viscous Hamilton-Jacobi equation
\begin{equation}\label{eq:nonlinear}
\begin{cases}
u_t+\lambda u-\mu(x_d) \Delta u=H(z,Du) \quad &\text{ in } \Omega_T,\\
u=0 \quad &\text{ on } (-\infty, T) \times \partial \bR^d_+,
\end{cases}
\end{equation}
where $\mu$ satisfies \eqref{con:mu} and $H:\Omega_T \times \bR^d \to \bR$ is a given Hamiltonian.
We assume that there exist $\beta, \ell >0$, and $h:\Omega_T \to {\overline{\bR_+}}$ such that, for all $(z,P)  \in \Omega_T \times \bR^d$,
\begin{equation} \label{G-cond}
 |H(z,P)| \leq{ \nu^{-1} (\min\{x_d^\beta,1\} |P|^{\ell}+x_d^\alpha h(z))}.
\end{equation}

The following is the main result in this section.
\begin{theorem} \label{example-thrm}
Let $p \in (1, \infty)$, $\alpha \in (0,2)$, and $\gamma \in (p(\alpha-1)_+-1, 2p-1)$.
Assume that \eqref{G-cond} holds with $\ell =1$, $\beta \geq 1$, and $h \in L_p(\Omega_T, x_d^\gamma\, dz)$.
Then, there exists $\lambda_0 = \lambda_0(d, p, \alpha, \beta, \gamma)>0$ sufficiently large such that the following assertion holds.
For any $\lambda \geq \lambda_0$, there exists a unique solution $u \in \sW^{1,2}_p(\Omega_T, x_d^\gamma \,dz)$ to \eqref{eq:nonlinear} such that
\[
\|\chi^{-\alpha} u_t\|_{L_p} + \|D^2 u\|_{L_p} +  \lambda \|\chi^{-\alpha} u\|_{L_p} \leq N \|h\|_{L_p}
\]
where $\|\cdot\|_{L_p} = \|\cdot \|_{L_p(\Omega_T, x_d^\gamma\, dz)}$ and $N = N(d, p, \alpha, \beta, \gamma)>0$.
\end{theorem}
\begin{proof}
The proof follows immediately from Theorem \ref{thm:xd} and the interpolation inequality in Lemma \ref{interpolation-inq} (i) below.
\end{proof}

\begin{remark}
Overall, it is meaningful to study \eqref{eq:nonlinear} for general Hamiltonians $H$.
It is typically the case that if we consider \eqref{eq:nonlinear} in $(0,T)\times \bR^d_+$ with a nice given initial data, then we can obtain Lipschitz a priori estimates on the solutions via the classical Bernstein method or the doubling variables method under some appropriate conditions on $H$.
See \cite{CIL, AT, LMT} and the references therein.
In particular, $\|Du\|_{L^\infty([0,T]\times \bR^d_+)} \leq N$, and hence, the behavior of $H(z,P)$ for $|P|>2N+1$ is unrelated and can be modified according to our purpose.
As such, if we assume \eqref{G-cond}, then it is natural to require that $\ell=1$ because of the above.

We note however that assuming \eqref{G-cond} with $\ell=1$ and $\beta \geq 1$ in Theorem \ref{example-thrm} is rather restrictive.
It is not yet clear to us what happens when $0\leq \beta<1$, and we plan to revisit this point in the future work.
\end{remark}

To obtain a priori estimates for solutions to \eqref{eq:nonlinear}, we consider the nonlinear term $H$ as a perturbation.
We prove the following interpolation inequalities when the nonlinear term satisfies \eqref{G-cond} with $\ell=1$ and $\ell=2$, which might be of independent interests.
\begin{lemma} \label{interpolation-inq} Let $p \in (1, \infty), \beta\ge 0, \gamma>-1$, $1 \leq \ell \leq \frac{d}{d-p}$, and $\theta = \frac{1}{2}(1+\frac{d}{p}-\frac{d}{\ell p})$. Assume that $H$ satisfies \eqref{G-cond}. The following interpolation inequalities hold for every $u \in C_0^\infty(\Omega_T)$ and $\tilde f(z) = x_d^{-\alpha} {\min\{x_d^\beta,1\}|Du|^\ell}$,
\begin{itemize}
\item[(i)] If $\ell=1$ and $\beta \geq 1$,
\[
\begin{split}
\|\tilde f\|_{L_p(\Omega_T, x_d^\gamma\, dz)} & \leq N  \|\chi^{-\alpha} u\|_{L_p(\Omega_T,x_d^\gamma\, dz)}^{1/2}\|D^2 u\|_{L_p(\Omega_T,x_d^\gamma\, dz)}^{1/2} \\
& \qquad +N \|\chi^{-\alpha} u\|_{L_p(\Omega_T,x_d^\gamma\, dz)},
\end{split}
\]
where $N = N(d, p, \beta, \gamma) >0$.
\item[(ii)] If $\ell=2$, $p \geq \frac{d}{2}$, and $\beta \geq \max\{\frac{\gamma}{p}+\frac{d\alpha}{2p}, \frac{\gamma}{p}+2 +\frac{\alpha}{d} - \frac{d\alpha}{p}\}$, then
\[
\begin{split}
\|\tilde f\|_{L_p(\Omega_T, x_d^\gamma\, dz)} & \leq N  \|\chi^{-\alpha} u\|_{L_p(\Omega_T,x_d^\gamma\, dz)}^{2(1-\theta)}\|D^2 u\|_{L_p(\Omega_T,x_d^\gamma\, dz)}^{2 \theta} \\
& \qquad +N \|\chi^{-\alpha} u\|_{L_p(\Omega_T,x_d^\gamma\, dz)}^2,
\end{split}
\]
where $N = N(d, p, \beta, \gamma) >0$.
\end{itemize}
\end{lemma}

\begin{proof}
For $m\in \bZ$, set $
\Omega_m=\{z\in \Omega_T\,:\, 2^{-m-1} < x_d \leq 2^{-m}\}$.
By the Gagliado-Nirenberg interpolation inequality, for $m\in \bZ$,
\[
\|Du\|_{L_{p\ell}(\Omega_m)} \leq N \left (\|u\|_{L_p(\Omega_m)}^{1-\theta} \|D^2u\|_{L_p(\Omega_m)}^\theta + 2^{2m\theta} \|u\|_{L_p(\Omega_m)} \right).
\]
Hence, for $m\geq 0$,
\begin{align*}
&\|\chi^{\beta-\alpha} |Du|^\ell\|_{L_p(\Omega_m,x_d^\gamma\, dz)}^p
=\int_{\Omega_m} x_d^{p(\beta-\alpha)+\gamma}|Du|^{p \ell}\,dz\\
&\leq \,  2^{-m(p(\beta-\alpha)+\gamma)} \int_{\Omega_m} |Du|^{p\ell}\,dz\\
&\leq \,  N2^{-m(p(\beta-\alpha)+\gamma)} \left(\int_{\Omega_m} |u|^{p}\,dz\right)^{\ell(1-\theta)} \left(\int_{\Omega_m} |D^2u|^{p}\,dz\right)^{\ell \theta} \\
&\qquad+ N2^{-m(p(\beta-\alpha)+\gamma+d-p\ell-d\ell)} \left(\int_{\Omega_m} |u|^{p}\,dz\right)^\ell\\
&\leq \,  N2^{-m(p(\beta-\alpha)+\gamma+ p \ell \alpha(1-\theta)-\ell\gamma )} \|\chi^{-\alpha} u\|_{L_p(\Omega_m,x_d^\gamma\, dz)}^{p \ell (1-\theta)}\|D^2 u\|_{L_p(\Omega_m,x_d^\gamma\, dz)}^{p \ell\theta} \\
&\qquad+ N2^{-m(p(\beta-\alpha)+\gamma+d-p \ell-d\ell+p \ell\alpha-\ell \gamma)}\|\chi^{-\alpha} u\|_{L_p(\Omega_m,x_d^\gamma\, dz)}^{p \ell}.
\end{align*}
By performing similar computations, we get that,  for $m< 0$,
\begin{align*}
&\|\chi^{-\alpha} |Du|^\ell \|_{L_p(\Omega_m,x_d^\gamma)}^p\\
&\leq \,  N2^{-m(-p\alpha+\gamma+ p \ell \alpha(1-\theta)-\ell \gamma )} \|\chi^{-\alpha} u\|_{L_p(\Omega_m,x_d^\gamma\, dz)}^{p \ell(1-\theta)}\|D^2 u\|_{L_p(\Omega_m,x_d^\gamma\, dz)}^{p \ell \theta} \\
&\qquad + N2^{-m(-p\alpha+\gamma+d-p\ell -d\ell +p\ell\alpha-\ell \gamma)}\|\chi^{-\alpha} u\|_{L_p(\Omega_m,x_d^\gamma\, dz)}^{p\ell}.
\end{align*}
Then, if $\ell=1$ and $\beta \geq 1$, we have
\[
\begin{cases}
p(\beta-\alpha)+\gamma+ p \ell \alpha(1-\theta)-\ell \gamma =p(\beta - \frac{\alpha}{2}) \geq 0,\\
p(\beta-\alpha)+\gamma+d-p\ell-d\ell+p\ell\alpha-\ell\gamma =p(\beta-1) \geq 0,\\
-p\alpha+\gamma+ p \ell \alpha(1-\theta)-\ell \gamma = -\frac{p\alpha}{2} \leq 0,\\
-p\alpha+\gamma+d-p\ell -d\ell +p\ell \alpha-\ell \gamma=-p \leq 0.
\end{cases}
\]
We thus obtain (i).
Similarly, the above four inequalities hold true when $\ell=2$, $p \geq \frac{d}{2}$, and $\beta \geq \max\{\frac{\gamma}{p}+\frac{d\alpha}{2p}, \frac{\gamma}{p}+2 +\frac{\alpha}{d} - \frac{d\alpha}{p}\}$, which yield (ii).
\end{proof}

\def\cprime{$'$}

\end{document}